\documentclass[journal,twoside,web]{ieeecolor}
\usepackage{generic}
\usepackage{cite}
\usepackage{amsmath,amssymb,amsfonts}
\usepackage{textcomp}
\usepackage{optidef}
 
\usepackage{amsthm}
\usepackage{xcolor}
\let\labelindent\relax
\usepackage{enumitem}
\usepackage{graphicx}
\usepackage{epstopdf}
\usepackage{algorithmic}
\usepackage{xcolor}
\usepackage{cite}
\usepackage{xurl}
\usepackage{hyperref}
\newtheorem{theorem}{Theorem}[section]
\newtheorem{lemma}[theorem]{Lemma}
\newtheorem{proposition}[theorem]{Proposition}
\newtheorem{definition}{Definition}[section]

\newtheorem{assumption}{Assumption}[section]
\newtheorem{remark}{Remark}[section]

\usepackage{amsopn}

\def\bbn{\mathbb N}
\def\bbz{\mathbb Z}
\def\bbr{\mathbb R}

\newcommand{\norm}[1]{\left\|#1 \right\|}

\newcommand{\absval}[1]{\mid {#1} \mid}

\def\BibTeX{{\rm B\kern-.05em{\sc i\kern-.025em b}\kern-.08em
    T\kern-.1667em\lower.7ex\hbox{E}\kern-.125emX}}
\begin{document}
\title{Distance Between Stochastic Linear Systems}
\author{Venkatraman Renganathan, \IEEEmembership{Member, IEEE}, and Sei Zhen Khong, \IEEEmembership{Senior Member, IEEE}
\thanks{This work was supported in part by the National Science and Technology Council of Taiwan (grant numbers: 113-2222-E-110-002-MY3, 114-2622-8-110-00 and 114-2218-E-007-011).}
\thanks{V. Renganathan is with the Faculty of Engineering \& Applied Sciences, Cranfield University, Bedfordshire, United Kingdom. Email: v.renganathan@cranfield.ac.uk. S. Z. Khong is with the Department of Electrical Engineering, National Sun Yat-sen University, Kaohsiung 80424, Taiwan. Email: szkhong@mail.nsysu.edu.tw.}
}

\maketitle

\begin{abstract}
While the existing stochastic control theory is well equipped to handle dynamical systems with stochastic uncertainties, a paradigm shift using distance measure based decision making is required for the effective further exploration of the field. As a first step, a distance measure between two stochastic linear time invariant systems is proposed here, extending the existing distance metrics between deterministic linear dynamical systems. In the frequency domain, the proposed distance measure corresponds to the worst-case point-wise in frequency Wasserstein distance between distributions characterising the uncertainties using inverse stereographic projection on the Riemann sphere. For the time domain setting, the proposed distance corresponds to the gap metric induced type-q Wasserstein distance between the distribution characterising the uncertainty of plant models. Apart from providing lower and upper bounds for the proposed distance measures in both frequency and time domain settings, it is proved that the former never exceeds the latter. The proposed distance measures will facilitate the provision of probabilistic guarantees on system robustness and controller performances.
\end{abstract}


\section{Introduction} \label{sec:introduction}
\IEEEPARstart{A}{ny} valid metric in a vector space will induce a topology which will then facilitate a rigorous mathematical construct for performing analysis in that metric space. Inspired by these basic facts, researchers in the early $1980$s aimed at constructing a valid distance metric between dynamical systems in the hope that this research direction will pave way for a mathematically feasible and provable robust control analysis. Predominantly, the following metrics have received a vast appreciation in the control community namely: 1) Gap metric \cite{zames1980unstable, el1985gap, georgiou1988computation, georgiou1989optimal}, 2) Graph metric \cite{vidyasagar1984graph}, and 3) $\nu$-Gap metric \cite{vinnicombe_tac_1993, vinnicombe2001uncertainty}. Authors in \cite{lanzon2009distance} proposed a generic notion of distance between systems that can be used to measure discrepancy between open-loop systems in a feedback sense under several uncertainty structures.
All the these metrics are equivalent to each other in the sense that they induce the same topology in the space of dynamical systems where closed loop stability happens to be a robust property. Such robust stability guarantees come with the presumption that all system models are equally probable in the considered neighbourhood set of plant models around the nominal plant model. However, nature is unpredictable while playing the role of an adversary inflicting uncertainties into the system dynamics and having an equally probable plants based assumption might downplay our ability to fully understand the nature's intention. In that sense, one can associate a probability distribution on the realization of the plants within the plant model ambiguity set in consideration. This initiated a research on probabilistic robust control using gap metric in \cite{RenganathanProbabilisticGap}. This paper is an extension along the lines of \cite{RenganathanProbabilisticGap, nystrom2024stereographic} but not with respect to plant models of the same stochastic system rather between two different stochastic systems in terms of their associated possible perturbed plant models. On a similar note, researchers in \cite{Fabrizio_ACC, calafiore2006scenario} also proposed several probabilistic robust control approaches to handle the nature violating the assumption on uncertainties with small probabilities.

Authors in \cite{afsari2013alignment} proposed a distance between two linear dynamical systems and called the alignment distance which is computed by finding the change of basis that best aligns the state-space realizations of the two linear dynamical systems. Similarly, authors in \cite{hanzon1982riemannian} came up with a Riemannian metric on the space of stable linear systems, with applications to identification problems. One of the prominent attempts in investigating distance between stochastic dynamical systems was done by authors in \cite{baggio2017lti}, where they came up with distance between spectral densities of linear time invariant (LTI) stochastic processes using behavioural theory. The theory of stochastic systems is not just limited to the field of mathematics but rather finds its application in many other fields of science. For instance, researchers in the field of medicine have started to think along in this direction too in \cite{sherlock2024distance} by coming up with an algorithmic approach to compute and identify appropriate distance metrics for the quantitative comparison of stochastic model outputs and time-evolving stochastic measurements of a system. Many researchers have analysed the robust performance of controllers in the robust control community through the lens of the distance metric theory. Our problem formulation with distance between stochastic LTI systems will eventually evolve towards analysing the resulting probabilistic robust performance of a stabilising controller of one of the stochastic LTI system. However, the main focus of this manuscript will only be on the proposal of an appropriate distance measure and obtaining bounds on them. The subsequent analysis with respect to the probabilistic robust performance is left as a future work. Similarly, adding probabilistic rigour on top of the associated robust stability analysis along the lines of \cite{khong2024robust, khong2013reconciling, zhao2021stabilization, su2025exponential} is also left as a future work. Our proposed research also has connections with frequency domain model validation problem considered in \cite{Abhishek_ACC_2013} where authors presented a frequency domain interpretation
of Monge-Kantorovich optimal transport. 

\emph{Contributions:} The main contributions of this paper are:
\begin{enumerate}
    \item We propose valid distance measures between two stochastic linear dynamical systems in the single input single output (SISO) case both in the frequency domain setting and in the time domain setting. 
    \item In the frequency domain setting, the proposed distance measure given by \eqref{eqn_distance_btw_p1_p2} refers to the worst-case point-wise-in-frequency type-$q$ chordal metric induced Wasserstein distance between distributions governing the uncertainties of the two stochastic linear systems in the Riemann sphere. Under the assumption of distributions being uniform in nature, a support distance based upper bound for proposed distance measure is given in Theorem \ref{theorem_support_distance_upper_bound}. On the other hand, using the deviation of the perturbed models of each systems from their respective nominal models, a lower bound for the proposed distance measure is given in Theorem \ref{theorem_triangle_lower_bound}. Extensions with respect to empirical distribution case are also given in Proposition \ref{proposition_empirical_distributions_distance} and Theorem \ref{theorem_triangle_lower_bound_time_domain_empirical}.
    \item In the time domain setting, the proposed distance measure given by \eqref{eqn_distance_between_distributions_system_models} refers to the gap metric induced type-$q$ Wasserstein distances between the distributions governing the uncertainty of the systems obtained through the push-forward of the distribution of uncertain system parameters under the measurable mapping that connects the parameters and the perturbed model. Upper bounds for the proposed distance measure are proposed in Proposition \ref{proposition_gap_diameter_upper_bound} and Theorem \ref{theorem_moment_upper_bound}. On the similar lines of the frequency domain setting, a lower bound for the proposed distance measure in the time domain is given using the deviation of the perturbed models of each systems from their respective nominal models in Proposition \ref{proposition_gap_lower_bound_using_nominal_gap}. 
    \item We also prove in Theorem \ref{theorem_distances_comparison} that for stochastic LTI systems, the proposed frequency domain distance never exceeds the time domain distance, mimicking the inequality relationship that exists between the $\nu$-gap metric and the gap metric in the deterministic systems setting.
    
\end{enumerate}

\subsection*{Paper Organisation}
Following a detailed discussion on the notations and the preliminaries needed for problem formulation in both frequency and time domain settings in Section \ref{sec:notations}, we will first begin our problem formulation in the frequency domain setting and propose a valid distance metric in Section \ref{sec:frequency_domain}. Following that, we will present the analogous problem formulation in the time domain setting and propose an analogous valid distance metric in Section \ref{sec:time_domain}. The proposed distance measures in the frequency and time domain settings are compared in Section \ref{sec_comparison}. Finally, the paper is closed in Section \ref{sec:conclusions} along with the summary of findings and directions for future research.
Simulation results are provided throughout the paper to demonstrate the proposed concepts and guarantees. All the Matlab codes responsible for reproducing the simulation results provided in the paper can be found at \url{https://github.com/venkatramanrenganathan/stochastic-system-gap}.

\section{Notations \& Preliminaries} \label{sec:notations}
The cardinality and closure of the set $A$ are denoted by $\left | A \right \vert$ and $\overline{A}$ respectively. The set of real numbers, integers and the natural numbers are denoted by $\bbr, \bbz, \bbn$ respectively and the subset of natural numbers greater than a given constant say $a \in \bbn$ is denoted by $\bbn_{> a}$. The Euclidean norm of a vector $x \in \mathbb{R}^{n}$ is denoted by $\norm{x}_{2}$ or simply $\norm{x}$. The inner product between two vector $r_1, r_2 \in \mathbb{R}^{n}$ is denoted by $\langle r_1, r_2 \rangle := r^{\top}_1 r_2$. For a matrix $A \in \bbr^{n \times n}$, we denote its transpose, trace, determinant, and the maximum (minimum) singular values by $A^{\top}$,  $\mathbf{Tr}(A)$, $\mathrm{det}(A)$, and $\overline{\sigma}(A) \,  (\underline{\sigma}(A))$ respectively. An identity matrix of dimension $n$ is denoted by $I_{n}$. The notation $(\cdot)_{+} := \max(0, \cdot)$ shall be used to ensure positivity. For brevity of notation, we shall be abbreviating functions $f(x(t), y(t))$ as $f(t; x, y)$. The composition of two functions $f, g$ is denoted by $f \circ g$.

\subsection{Function Spaces \& Norms}
The space of complex numbers is denoted by $\mathbb{C}$ and $j$ represents the imaginary unit. For a complex variable $z \in \mathbb{C}$, we denote its complex conjugate as $z^{\star} \in \mathbb{C}$. Let $\mathbf{R}(s)$ denote the set of rational functions in $s \in \mathbb{C}$ with real coefficients. We use $\mathcal{P}(s) \subset \mathbf{R}(s)$ to denote the set of proper rational functions whose poles are in the open left half-plane. Let us denote the set of matrices with elements in $\mathbf{R}(s)$ as $\mathrm{mat}(\mathbf{R}(s))$ and similarly let us denote the set of matrices with elements in $\mathcal{P}(s)$ as $\mathrm{mat}(\mathcal{P}(s))$. A continuous-time signal $x \in \mathbb{R}^{n}$ is said to be in $\mathcal{L}_{2}$ space if it has bounded energy. Let $\mathcal{H}_{2}$ denote the space of Fourier transform of signals in $\mathcal{L}_{2}$ space but restricted to positive time. Dynamical systems are to be considered as operators on $\mathcal{H}_{2}$ and they will be called \emph{stable} if for any input $u \in \mathcal{H}_{2}$, the system output $y \in \mathcal{H}_{2}$. The Hardy space consisting of transfer functions of stable LTI continuous time systems is denoted by $\mathcal{H}_{\infty}$ and is equipped with the norm
\begin{align}\label{eqn_h_infty_norm_def}
    \norm{P}_{\mathcal{H}_{\infty}} 
    := \sup_{\substack{u \in \mathcal{H}_{2} \\ u \neq 0}} \frac{\norm{Pu}_{\mathcal{H}_{2}}}{\norm{u}_{\mathcal{H}_{2}}}.
\end{align}
Let $\mathbf{R}\mathcal{H}_{\infty} = \mathbf{R}(s) \cap \mathcal{H}_{\infty}$ and similarly, $\mathbf{R}\mathcal{L}_{\infty} = \mathbf{R}(s) \cap \mathcal{L}_{\infty}$, where $\mathcal{L}_{\infty}$ is the space of all functions that are essentially bounded on the imaginary axis with norm
\begin{align}
\norm{P}_{\mathcal{L}_{\infty}} 
:=
\mathop{\mathrm{ess\,sup}}_{\omega \in \mathbb{R}}
\bar{\sigma}(P(j \omega)).
\end{align}
The probability space is defined using a triplet $(\Omega, \mathcal{F}, \mathbb{P})$, where $\Omega, \mathcal{F}$, and $\mathbb{P}$ denote the sample space, event space and the probability function respectively. 
A probability distribution with mean $\mu$ and covariance $\Sigma$ is denoted by $\mathbb{P}(\mu, \Sigma)$ and, specifically $\mathfrak{N}_{d}(\mu, \Sigma)$, if the distribution is normal in $\mathbb{R}^{d}$. A real random vector $x \in \mathbb{R}^{n}$ following a distribution $\mathbf{f}_{x}$ is denoted by $x \sim \mathbf{f}_{x}$.
A complex random variable $Z$ on the probability space $(\Omega, \mathcal{F}, \mathbb{P})$ is a function $Z: \Omega \rightarrow \mathbb{C}$ such that both its real part 
$\mathfrak{R}(Z)$ and its imaginary part 
$\mathfrak{I}(Z)$ are real random variables on $(\Omega, \mathcal{F}, \mathbb{P})$. 
A uniform distribution defined over a compact set $A$ is denoted by $\mathcal{U}(A)$. 
Given  $q \ge 1$, the set of probability measures in $\mathcal{P}(\bbr^{d})$ with finite $q$\textsuperscript{th} moment is denoted by $\mathcal{P}_{q}(\bbr^{d}) := \left\{ \mu \in \mathcal{P}(\bbr^{d}) \mid \int_{\bbr^{d}} \left \Vert x \right \|^{q} \mathrm{d}\mu < \infty \right\}$. 
The $\mathcal{L}_{q}$ norm of a random variable $x \in \mathbb{R}$ with $x \sim \mathbf{f}_{x}$ is denoted by $\norm{x}_{\mathcal{L}_{q}} := \left(\int x^{q} d\mathbf{f}_{x}\right)^{1/q}$. If not specified, the notation $\norm{x}$ simply denotes the $\mathcal{L}_{2}$ norm of the random variable $x$.
The type-$q$ Wasserstein distance $\forall q \geq 1$ between distributions $\mathbb{Q}_1, \mathbb{Q}_2 \in \mathcal{P}_{q}(\bbr^{d})$ with $\Pi(\mathbb{Q}_1, \mathbb{Q}_2)$ being the set of all joint distributions on $\bbr^{d} \times \bbr^{d}$ with marginals $\mathbb{Q}_1$ and $\mathbb{Q}_2$ is given by
\begin{equation} \label{eqn_wass_distance}
W^{q}_{q}(\mathbb{Q}_1, \mathbb{Q}_2) 
\overset{\Delta}{=}  
\underset{\pi \in \Pi(\mathbb{Q}_1, \mathbb{Q}_2)}{\inf}  \int_{\bbr^{d} \times \bbr^{d}} \left \Vert z_1 - z_2 \right \|^{q} \pi(d z_{1}, d z_{2}). 
\end{equation}
\begin{definition}\label{define_pushforward}
Given measurable spaces $(X,\mathcal{X})$, and $(Y,\mathcal{Y})$, let $f: X \to Y$ be a measurable map, and $\mu$ a probability measure on $(X,\mathcal{X})$. Then, for any Borel set $\mathbf{B} \subset \mathcal{Y}$, the push-forward measure $f_{\#}\mu$ on $(Y,\mathcal{Y})$ is defined as
\begin{align}
\label{eqn_pushforward_definition}
(f_{\#}\mu)(\mathbf{B}) 
:= 
\mu\left(f^{-1}(\mathbf{B})\right),
\end{align}
\end{definition}
where $f^{-1}(\mathbf{B})$ denotes the pre-image of $\mathbf{B} \subset \mathcal{Y}$ in $\mathcal{X}$ and is defined as
\begin{align}
f^{-1}(\mathbf{B}) 
:=
\left\{
x \in \mathcal{X} \mid 
f(x) \in \mathcal{B}
\right\}.
\end{align}
\subsection{Preliminaries on Stereographic Projections}
To understand the development of the distance measure proposed for the frequency domain setting in this manuscript, we define the Riemann sphere to model the extended complex plane \cite{needham2023visual}. 
\begin{definition}
The Riemann sphere, denoted by $\mathfrak{R} \subset \mathbb{R}^{3}$ is a sphere centred at $\mathfrak{R}_{c} := \left(0,0,\frac{1}{2}\right)$ with unit diameter tangent at its south pole to $\mathbb{C}$ at the origin and its boundary is denoted by $\partial \mathfrak{R}$. That is,
\begin{align}
\label{eq:Riemann_sphere}
    \mathfrak{R}
    =
    \left\{
    (x,y,z) \in \mathbb{R}^{3} \mid x^{2} + y^{2} + \left(z-0.5 \right)^{2} = \left(0.5\right)^{2}
    \right\}.
\end{align}
\end{definition}
We can also express the $\mathfrak{R}$ in terms of spherical coordinates using the point $(x,y,z)=\left(0, 0, \frac{1}{2}\right)$ as its origin. The equivalent coordinates would then be $\left(\frac{1}{2},\theta,\varphi\right)$ where $\theta$ and $\varphi$ are the polar and azimuthal angles respectively. Then, we can alternatively represent $\mathfrak{R}$ as
\begin{equation}\label{eq:Riemann_spherical_coords}
    \mathfrak{R}
    =
    \left\{
    (r,\theta,\varphi)\in\mathbb{R}^{3} \mid  r = \frac{1}{2},\ \theta\in[0, \pi],\ \varphi\in[0, 2\pi]
    \right\}.
\end{equation}
We now define the stereographic projection of points onto the complex plane $\mathbb{C}$ from the Riemann sphere, with its north and south pole denoted by $\mathfrak{N}$ and $\mathfrak{S}$ respectively.
\begin{definition}
Let $R = (x, y, z) \in \partial \mathfrak{R} \backslash \{\mathfrak{N}\}$. Then, a line through $\mathfrak{N}$ and $R$ intersects $\mathbb{C}$ exactly at one point $\phi(R) \in \mathbb{C}$. That is, $\phi: \partial \mathfrak{R} \backslash \{\mathfrak{N}\} \mapsto \mathbb{C}$ and the association from $R \mapsto \phi(R)$ is called the stereographic projection.
\end{definition}
Given a point $R = (r_x, r_y, r_z) \in \partial \mathfrak{R} \backslash \mathfrak{N}$, the coordinates of the corresponding Stereographic projected point $c \in \mathbb{C}$ is given by
\begin{align}
\label{eqn_complexplane_cartesian_coords}
    c 
    = 
    \phi(R) 
    :=
    \left(
    \frac{r_x}{1 - r_z} 
    \right)
    + j 
    \left(
    \frac{r_y}{1 - r_z}
    \right).
\end{align}
We will also require the inverse of the stereographic projection. 
\begin{definition}
Let $\phi(R) \in \mathbb{C}$ denote a point on the complex plane. Then, the line from $\phi(R) \in \mathbb{C}$ to the $\mathfrak{N}$ intersects $\partial \mathfrak{R}$ exactly at one point $R = (x, y, z) \in \partial \mathfrak{R}$. That is, $\phi^{-1} : \mathbb{C} \rightarrow \partial \mathfrak{R}$ and the association from $\phi(R) \mapsto R$ is called the inverse of the stereographic projection. 
\end{definition}
Given a point $c \in \mathbb{C}$, the corresponding point $R := \phi^{-1}(c) = (r_x, r_y, r_z) \in \partial \mathfrak{R}$ due to the inverse of the stereographic projection will have the Cartesian coordinates:
\begin{align} \label{eqn_sphere_cartesian_coords}
(r_x, r_y, r_z) 
= 
\left( 
\frac{\mathrm{Re}(c)}{1 + \absval{c}^{2}}, 
\frac{\mathrm{Im}(c)}{1 + \absval{c}^{2}}, 
\frac{\absval{c}^{2}}{1 + \absval{c}^{2}} 
\right).
\end{align}
The following proposition will precisely characterize the distribution transformation under the inverse of the stereographic projection operation using results from \cite{Geometric_Measure_Book, Papoulis_Book}. 
\begin{proposition}
\label{proposition_projected_distribution}
(From \cite{Anton_MS_Thesis}) Let $P = \mathsf x + j \mathsf y $ be the random frequency response at a frequency $\omega \in \Omega$ and let the distribution characterizing the uncertainty of 
$P$ in $\mathbb{C}$ along the real and imaginary axes denoted by $\mathbb{P}_\mathsf{xy}(x,y)$ be known apriori. Then, the distribution $\mathbb{P}_{R}$ that characterizes the corresponding uncertainty on the Riemann sphere due to the stereographic projection mapping is given by
\begin{align}
\label{eq:projected_distribution}
\mathbb{P}_{R} 
&= 
\frac{r + r^{3}}{2} \mathbb{P}_{\mathsf{xy}}(x,y), \quad \text{where } r = \sqrt{x^2 + y^2}.
\end{align}
\end{proposition}

\subsection{Preliminaries on Gap Metric}
Towards our distance measure definition on the time domain setting, we provide here some preliminary details. Particularly, we will start with some basics regarding the concepts needed to understand the gap metric. 
Every matrix $P(s) \in \mathrm{mat}(\mathbf{R}(s))$ has both a Right Co-prime Factorisation (RCF) as well as a Left Co-prime Factorisation (LCF) over the ring $\mathcal{P}(s)$. That is, $\forall P(s) \in \mathrm{mat}(\mathbf{R}(s))$, there exist $N, D, \Tilde{N}, \Tilde{D}, X, Y, \Tilde{X}, \Tilde{Y} \in \mathrm{mat}(\mathcal{P}(s)) \cap \mathcal{H}_{\infty}$ such that 
\begin{align}
\label{eqn_coprime_representation}
P
= 
N D^{-1} 
= 
\Tilde{D}^{-1}\Tilde{N},    
\end{align}
and the following \emph{Bezout's identity} holds for all $s \in \mathbb{C}_{\geq 0}$, 
\begin{align*}
X(s)N(s) + Y(s)D(s) = \Tilde{N}(s) \Tilde{X}(s) + \Tilde{D}(s) \Tilde{Y}(s) = I.
\end{align*} 
Further, the RCF is said to be \emph{normalized} if in addition it satisfies $N^{\star} N + D^{\star} D = I$. Analogous LCF results are available and are omitted here for the reason of being not used in this paper. Given $P(s) \in \mathrm{mat}(\mathbf{R}(s))$, its $\mathcal{H}_2$-graph is defined as
\begin{subequations}
\label{eqn_graph_p}
\begin{align}
    \mathcal{G}_{P} 
    &:= 
    \left\{ (u,y) \mid y = Pu \right\} \subseteq \mathcal{H}_{2} \times \mathcal{H}_{2} \\
    &=
    \underbrace{
    \begin{bmatrix}
    D \\ N    
    \end{bmatrix}}_{=: G}
    \mathcal{H}_{2} = \mathrm{Range}(G),
\end{align}
\end{subequations}
where the operator $G$ (henceforth referred to as the graph symbol) is unitary meaning that $G^{\star} G = I$. Note that $\mathcal{G}_{P}$ is a closed subspace of $\mathcal{H}_{2} \times \mathcal{H}_{2}$. The orthogonal projection onto $\mathcal{G}_{P}$ is denoted by $\Pi_{\mathcal{G}_{P}}$ and it is bounded. Let $G_{1}$ and $G_{2}$ denote the graph symbols of normalized RCFs of plants $P_{1}$ and $P_{2}$ respectively. The gap between the systems (interested authors are referred to \cite{georgiou1988computation, RenganathanProbabilisticGap} and the references therein) $P_1$ and $P_2$ can be defined as
\begin{align}
\label{eqn_gap_metric_definition}
\delta_{g}(P_1, P_2) = \left\| \Pi_{\mathcal{G}_1} - \Pi_{\mathcal{G}_2} \right\|.
\end{align} 
Given a plant $P$ and a stabilising controller $C$ for it, we denote the associated performance measure as $b_{P,C} \in (0,1)$ and is defined as 
\begin{align}
\label{eqn_bPC_definition}
b_{P,C}
=
\norm{
\begin{bmatrix}
P \\ I    
\end{bmatrix}
(I - CP)^{-1}
\begin{bmatrix}
-C & I    
\end{bmatrix}
}^{-1}_{\mathcal{H}_{\infty}}
\end{align}

\section{Problem Formulation in Frequency Domain} \label{sec:frequency_domain}
While our main aim is to come up with an appropriate frequency domain specific distance metric between stochastic dynamical systems in general, for the ease of exposition, we shall start the problem formulation by analysing simple single input single output (SISO) dynamical systems first. The exposition with multiple input multiple output (MIMO) systems is out of the scope of this manuscript and is being investigated as a part of our future ongoing research (though we believe that exposition should carry forward typically from SISO to MIMO). Let $\Omega = [0, \infty)$ denote the set of all frequencies. \\

Consider the setting where two stochastic SISO LTI dynamical systems, each with $n \in \mathbb{N}$ states, $m \in \mathbb{N}$ control inputs, and $l \in \mathbb{N}$ outputs, in the space of $\mathbf{R}\mathcal{L}_{\infty}$ are given and we denote their transfer functions as $P_{1}(s)$ and $P_{2}(s)$ respectively. Let $\theta_{i} \sim \mathbf{f}_{\theta_{i}} = \mathcal{N}(\mu_{\theta_{i}}, \Sigma_{\theta_{i}})$ denote the random parameter affecting the system $i \in \{1, 2\}$ with $p \leq (n^{2} + mn + ln)$. Then, the transfer function of the $i$\textsuperscript{th} stochastic system for a fixed $s \in \mathbb{C}$ can be written as
\begin{align}
\label{eqn_stochastic_transfer_function}
P_{i}(\theta_{i}; s) 
=
C_{i}(\theta_{i}) (sI - A_{i}(\theta_{i}))^{-1} B_{i}(\theta_{i}),
\end{align}
where the matrices $A_{i}(\theta_{i}) \in \mathbb{R}^{n \times n}, B_{i}(\theta_{i}) \in \mathbb{R}^{n \times m}, C_{i}(\theta_{i}) \in \mathbb{R}^{l \times n}$. It is evident from \eqref{eqn_stochastic_transfer_function} that the randomness in the parameter $\theta_{i} \sim \mathbf{f}_{\theta_{i}} = \mathcal{N}(\mu_{\theta_{i}}, \Sigma_{\theta_{i}})$ manifests itself to render a random transfer function $P_{i}(\theta_{i}; s)$. For brevity of notation, we will write $P_{i}(\theta_{i}; s)$ simply as $P_{i}(s)$. 

\begin{assumption}
The frequency responses of both the systems at any given frequency $\omega \in \Omega$ denoted by $P_{1}(j \omega)$ and $P_{2}(j \omega)$ are random and governed by distributions $\mathbb{P}_{P_{1}(\omega)}$ and $\mathbb{P}_{P_{2}(\omega)}$ respectively. That is, $P_{1}(j \omega) \sim \mathbb{P}_{P_{1}(\omega)}$ and $P_{2}(j \omega) \sim \mathbb{P}_{P_{2}(\omega)}$.
\end{assumption}
Now, $\forall \omega \in \Omega$, let us define the measurable map $\mathcal{E}_{\omega}: \mathbf{R}\mathcal{L}_{\infty} \rightarrow \mathbb{C}$ such that $\mathcal{E}_{\omega}(P(s)) = P(j\omega)$. Then, the distribution $\mathbb{P}_{P_{i}(\omega)}$ governing the uncertainty of system $i$ on $\mathbb{C}$ with $i \in \{1, 2\}$ is related to the distribution $\mathbf{f}_{\theta_{i}}$ of the random parameter $\theta_{i}$ as
\begin{align}
\label{eqn_distribution_pushfwd_complexplane}
\mathbb{P}_{P_{i}(\omega)}
:=
\left(
\Phi^{\mathbb{C}}_{i}
\right)_{\#}
\mathbf{f}_{\theta_{i}}, 
\quad
\text{where}
\quad
\Phi^{\mathbb{C}}_{i}
=
\mathcal{E}_{\omega} \circ P_i,
\end{align}
and for any Borel measurable set $\mathbf{B} \subset \mathbb{C}$, the above push-forward measure satisfies
\begin{align}
\label{eqn_pushforward_measure_frequency_domain}
\left(
\Phi^{\mathbb{C}}_{i}
\right)_{\#} 
\mathbf{f}_{\theta_{i}}(\mathbf{B}) 
 =
 \mathbb{P}( \Phi^{\mathbb{C}}_{i} \in \mathbf{B})
 =
 \mathbf{f}_{\theta_{i}}
 \left(
 \left(
 \Phi^{\mathbb{C}}_{i}
 \right)^{-1}(\mathbf{B})
 \right).
\end{align}


\begin{assumption}
The support sets of the distributions $\mathbb{P}_{P_{1}(\omega)}$ and $\mathbb{P}_{P_{2}(\omega)}$ denoted by $\mathcal{S}_{P_{1}}(\omega)$ and $\mathcal{S}_{P_{2}}(\omega)$ respectively are both convex and compact.   
\end{assumption}
\begin{remark}
Compactness is essential to exclude the case of $\infty$ being included in the support set. Notice that both the distributions and their corresponding support sets of both the plants are frequency-dependent. This modelling assumption makes sense as one usually performs system identification procedure to identify plant models for a system by exciting the system at all frequencies using appropriate input signals. However, notice that we do not make any explicit assumption on the support sets $\mathcal{S}_{P_{1}}(\omega)$ and $\mathcal{S}_{P_{2}}(\omega)$ being disjoint from each other.
\end{remark}
\begin{assumption}
For every $\omega \in \Omega$, a nominal plant model for both the systems $P_1$ and $P_2$ denoted by $\bar{P}_{1}(j\omega)$ and $\bar{P}_{2}(j\omega)$ respectively are known apriori.     
\end{assumption}
From now on, we shall drop the $(j \omega)$ argument for convenience with the understanding that the formulation corresponds to the quantities at a particular frequency $\omega$ unless otherwise specified. For both the stochastic systems enumerated by $\ell = 1,2$, we can infer its projected support set as
\begin{align}
\mathcal{R}_{P_{\ell}} 
\!=\! 
\phi^{-1}(\mathcal{S}_{P_{\ell}}) 
\!=\!
\left\{\!
R \in \partial \mathfrak{R} 
\mid
R = \phi^{-1}(r), \forall r \in \mathcal{S}_{P_{\ell}} \!
\right\}.
\end{align}

\subsection{Stereographic Projection of Distribution}
To find the distance between the random plants, we first need to understand how their corresponding distributions get transformed under the stereographic projection operation. 
That is, we need to characterise how the distributions $\mathbb{P}_{P_{\ell}(\omega)}$ of system $\ell = \{1, 2\}$ will get transformed under the inverse of the stereographic projection operation. We recall Proposition \ref{proposition_projected_distribution} and use 
\eqref{eq:projected_distribution} to obtain the corresponding projected distribution $\mathbb{P}_{R_{\ell}(\omega)}$ living on the Riemann sphere, for each system $\ell = 1,2$ due to the inverse of the stereographic projection mapping. An illustration is provided in Figure \ref{fig:2stereos}.
\begin{figure}[h]
    \centering
    \includegraphics[width=0.5\linewidth]{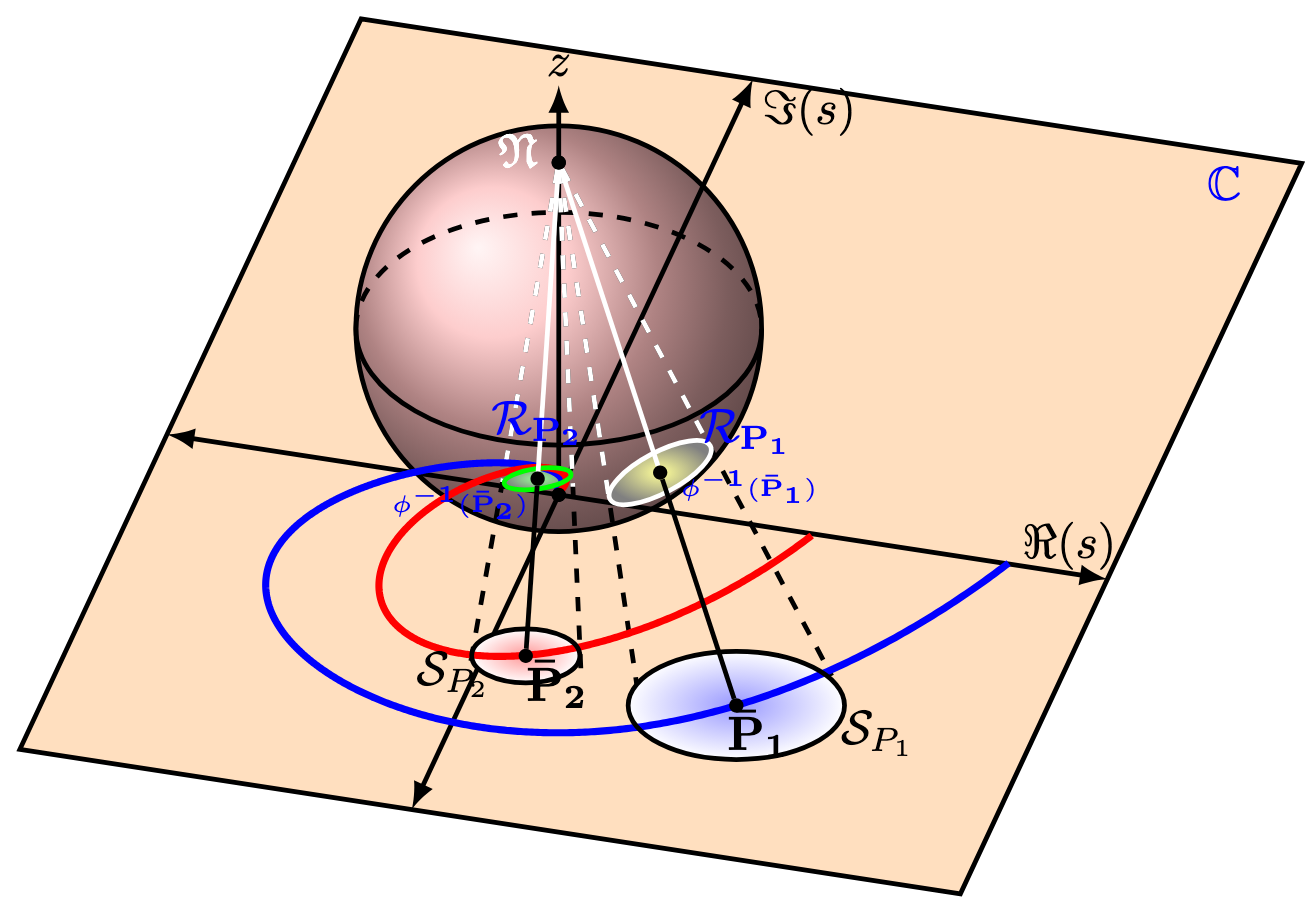}
    \caption{The Riemann sphere tangent to $\mathbb{C}$ is shown in shaded dark brown. An instance of the two stochastic systems $P_1$ and $P_2$ are depicted using their Nyquist plot in blue \& red curves respectively. The distributions $\mathbb{P}_{P_{1}}$ and $\mathbb{P}_{P_{2}}$ characterizing the uncertainties of $P_1$ and $P_2$ at a frequency are shown as shaded blue \& red colours with compact support sets $\mathcal{S}_{P_{1}}$ and $\mathcal{S}_{P_{2}}$ in $\mathbb{C}$ respectively. The corresponding inverse stereographic projections of the support sets onto the Riemann sphere are shown as sets $\mathcal{R}_{P_{1}}$ and $\mathcal{R}_{P_{2}}$ respectively. The known nominal models $\bar{P}_{1}, \bar{P}_{2}$ along with their projected counterparts on the Riemann sphere $\phi^{-1}(\bar{P}_{1}), \phi^{-1}(\bar{P}_{2})$ are also shown here.}
    \label{fig:2stereos}
\end{figure}
Equivalently, the distribution $\mathbb{P}_{R_{i}(\omega)}$ governing the uncertainty of system $i$ with $i \in \{1, 2\}$ on the boundary of the Riemann sphere is related to the distribution $\mathbf{f}_{\theta_{i}}$ of the random parameter $\theta_{i}$ through the push-forward relation as
\begin{align}
\label{eqn_distribution_pushfwd_Riemannsphere}
\mathbb{P}_{R_{i}(\omega)}
:=
\left(
\Phi^{\mathfrak{R}}_{i}
\right)_{\#} 
\mathbf{f}_{\theta_{i}},
\quad 
\text{where}
\quad 
\Phi^{\mathfrak{R}}_{i} = \phi^{-1} \circ \Phi^{\mathbb{C}}_{i},
\end{align}
and for any Borel measurable set $\mathbf{B} \subseteq \partial \mathfrak{R}$, the above push-forward measure satisfies
\begin{align}
\label{eqn_pushforward_measure_RiemannSphere}
\left(
\Phi^{\mathfrak{R}}_{i}
\right)_{\#} 
\mathbf{f}_{\theta_{i}}(\mathbf{B}) 
=
\mathbb{P}( \Phi^{\mathfrak{R}}_{i} \in \mathbf{B})
=
\mathbf{f}_{\theta_{i}}
\left(
\left(
\Phi^{\mathfrak{R}}_{i}
\right)^{-1}(\mathbf{B})
\right).
\end{align}
\subsection{Support Distance Between Systems $P_1$ and $P_2$}
In order to understand how far two stochastic dynamical systems $P_1$ and $P_2$ are in the frequency domain, first we analyse the distance between their support sets $\mathcal{S}_{P_{1}}(\omega)$ and $\mathcal{S}_{P_{2}}(\omega)$, where the respective system realizations can occur for every frequency $\omega \in \Omega$. To this end, we define the support distance between $P_1$ and $P_2$ in $\mathbb{C}$.

\begin{definition}
Given two systems $P_{\ell}$ with $\ell = \{1, 2\}$, whose uncertainty in $\mathbb{C}$ are characterized by distributions defined on support sets $\mathcal{S}_{P_{\ell}}(\omega)$ for every frequency $\omega \in \Omega$, we define the support distance between the systems as the worst-case pointwise-in-frequency distance between points in their support sets $\mathcal{S}_{P_{\ell}}(\omega)$. That is, 
\begin{align}
\label{eqn_distance_btw_s}
\mathrm{d^{\mathbb{C}}_{sup}}(P_1, P_2)
&:=
\sup_{\omega \in \Omega} \, 
\sup_{
\substack{
s_1 \in \mathcal{S}_{P_{1}}(\omega) \\ 
s_2 \in \mathcal{S}_{P_{2}}(\omega)
}
}
\norm{s_1 - s_2}.
\end{align}
\end{definition}
Since in the complex plane, the Euclidean distance between points can be more than $1$, we would like to infer the support distance in the Riemann sphere where distance shall never exceed the value of $1$. Hence, a similar definition between the support sets of distributions living in the Riemann sphere due to the inverse of the stereographic projection operation can also be defined. Since in that case the distributions will live on the Riemann sphere, we need to use the appropriate metric to facilitate the Wasserstein distance computation having the optimal transport perspective taking into perspective the curvature of the underlying manifold. We state a basic result from the differential geometry \cite{do2016differential}. 
\begin{proposition}
Given two points $r_1 \in \mathcal{R}_{P_{1}}(\omega)$, and $r_2 \in \mathcal{R}_{P_{2}}(\omega)$, the geodesic distance between them, denoted by $\mathrm{d_{geo}}(r_1, r_2)$ and the chordal distance between them, denoted by $\mathrm{d_{chord}}(r_1, r_2)$ are respectively  given by
\begin{align}
\mathrm{d_{geo}}(r_1, r_2)
&:=
\frac{1}{2} \cos^{-1}\left(4 \langle (r_1 - \mathfrak{R}_{c}), (r_2 - \mathfrak{R}_{c}) \rangle \right), \label{eqn_geodesic_distance_definition}   \\
\mathrm{d_{chord}}(r_1, r_2)
&:=
\norm{r_1 - r_2} 
=
\sin(\mathrm{d_{geo}}(r_1, r_2)), \quad \text{and} \label{eqn_chordal_distance_definition}   \\
\mathrm{d_{chord}}(r_1, r_2)
&\leq 
\mathrm{d_{geo}}(r_1, r_2). \label{eqn_chordal_geo_relation}
\end{align}
\end{proposition}
\begin{proof}
For the Riemann sphere centred at $\mathfrak{R}_{c} = (0,0,0.5)$ and of radius $R = 0.5$, let $\theta$ be the central angle measured at $\mathfrak{R}_{c}$ between the points $r_1 \in \mathcal{R}_{P_{1}}(\omega)$ and $r_2 \in \mathcal{R}_{P_{2}}(\omega)$. Then, $\mathrm{d_{geo}}(r_1, r_2) = R \theta$ which evaluates to \eqref{eqn_geodesic_distance_definition}. Note that $\mathrm{d_{geo}}(r_1, r_2) = R \theta \iff \theta = \frac{\mathrm{d_{geo}}(r_1, r_2)}{R} = 2 \mathrm{d_{geo}}(r_1, r_2)$. Similarly, $\mathrm{d_{chord}}(r_1, r_2) = 2R \sin\left(\frac{\theta}{2}\right)$ which then evaluates to \eqref{eqn_chordal_distance_definition}. Since $\sin(x) \leq x, \forall x \geq 0$, the result \eqref{eqn_chordal_geo_relation} follows.
\end{proof}
With this chordal metric, we can now define the support distance on the Riemann sphere between the systems $P_{\ell}$ for $\ell = \{1, 2\}$.
\begin{definition}
Given two systems $P_{\ell}$ with $\ell = \{1, 2\}$, whose uncertainty in $\mathbb{C}$ are characterized by distributions defined on support sets $\mathcal{S}_{P_{\ell}}(\omega)$ for every frequency $\omega \in \Omega$, the support distance in the Riemann sphere between the systems is defined as the worst-case pointwise-in-frequency chordal distance between points in their projected support sets $\mathcal{R}_{P_{\ell}}(\omega)$. That is,
\begin{align}
\label{eqn_distance_btw_r}
\mathrm{d^{\mathfrak{R}}_{sup}}(P_1, P_2)
&:=
\sup_{\omega \in \Omega} \, 
\underbrace{
\sup_{
\substack{
r_1 \in \mathcal{R}_{P_{1}}(\omega) \\ 
r_2 \in \mathcal{R}_{P_{2}}(\omega)
}
}
\mathrm{d_{chord}}(r_1, r_2)}_{:= \mathrm{d^{\mathfrak{R}}_{sup}}(P_1, P_2, \omega)}.
\end{align}
\end{definition}
Clearly, both the support distances $\mathrm{d^{\mathbb{C}}_{sup}}(P_1, P_2)$ and $\mathrm{d^{\mathfrak{R}}_{sup}}(P_1, P_2)$ given by \eqref{eqn_distance_btw_s} and \eqref{eqn_distance_btw_r} do not explicitly take into account the information of the distributions $\mathbb{P}_{P_{1}(\omega)}$ and $\mathbb{P}_{P_{2}(\omega)}$ and rather is based only on their support sets $\mathcal{S}_{P_{1}}(\omega)$ and $\mathcal{S}_{P_{2}}(\omega)$ respectively for every frequency $\omega \in \Omega$. That is, both the support distances $\mathrm{d^{\mathbb{C}}_{sup}}(P_1, P_2)$ and $\mathrm{d^{\mathfrak{R}}_{sup}}(P_1, P_2)$ given by \eqref{eqn_distance_btw_s} and \eqref{eqn_distance_btw_r} do not take into account the frequency with which each plant models occur within their respective support sets meaning that it just encodes the physical separation between the plant models getting realised in the respective support sets. This shortcoming can be addressed by defining a chordal metric induced Wasserstein distribution between distributions living in the $\partial \mathfrak{R}$.

\subsection{Type-$q$ Distance Between Systems $P_1$ and $P_2$}
We address the above shortcoming by proposing the type-$q$ Wasserstein distance between systems by taking into account their distributions $\mathbb{P}_{P_{1}(\omega)}$, $\mathbb{P}_{P_{2}(\omega)}$ for every frequency $\omega \in \Omega$. We denote the corresponding set of all possible joint distribution by $\Pi_{\omega} := \Pi\left(\mathbb{P}_{R_{1}(\omega)}, \mathbb{P}_{R_{2}(\omega)}\right)$.

\begin{definition}
Given $q \geq 1$ and two systems $P_1$ and $P_2$ whose uncertainties in $\mathbb{C}$ are characterized by distributions $\mathbb{P}_{P_{1}(\omega)}$ and $\mathbb{P}_{P_{2}(\omega)}$ defined on support sets $\mathcal{S}_{P_{1}}(\omega)$ and $\mathcal{S}_{P_{2}}(\omega)$ respectively $\forall \omega \in \Omega$, we define the type-$q$ distance between the systems as the worst-case point-wise-in-frequency type-$q$ chordal metric induced Wasserstein distance between their projected distributions $\mathbb{P}_{R_{1}(\omega)}$ and $\mathbb{P}_{R_{2}(\omega)}$ defined on support sets $\mathcal{R}_{P_{1}}(\omega)$ and $\mathcal{R}_{P_{2}}(\omega)$ respectively. That is,    
\begin{align}
\label{eqn_distance_btw_p1_p2}
&\mathrm{d}_{q}\left(P_1, P_2\right) \nonumber \\
&:=
\sup_{\omega \in \Omega} 
W^{q}_{q}\left(\mathbb{P}_{R_{1}(\omega)}, \mathbb{P}_{R_{2}(\omega)}\right) \nonumber \\
&=
\sup_{\omega \in \Omega} \,
\underset{\pi_{\omega} \in \Pi_{\omega}}{\inf}  \int_{\mathcal{R}_{P_{1}}(\omega) \times \mathcal{R}_{P_{2}}(\omega)} 
\mathrm{d_{chord}}(r_1, r_2)^{q} \, 
\pi_{\omega}(d r_{1}, d r_{2}).
\end{align}
\end{definition}
\begin{remark}
One can use other variations to define the distance between systems $P_1$ and $P_2$ using other distance measures such as total variation measure, Hellinger measure, $\chi^{2}$ measure to measure their point-wise-in-frequency distance between the distributions $\mathbb{P}_{R_{1}(\omega)}$ and $\mathbb{P}_{R_{2}(\omega)}$ at every frequency $\omega \in \Omega$. Each comes with its own merits and drawbacks. We will stick to the Wasserstein distance based definition for this manuscript.     
\end{remark}
\begin{remark}
Note that the $\nu$-gap metric from the robust control literature is defined using the chordal distance between points on the Riemann sphere obtained through the inverse stereographic projection, provided that the two systems satisfy certain winding number constraints. It is certainly possible to take into account the Riemann manifold and use the geodesic distance as the transport cost while computing the Wasserstein distance between distributions on the Riemann sphere. In such a case, the optimal transport plan shall happen along the boundary of the Riemann sphere and as a result the geodesic metric $\mathrm{d_{geo}}(r_1, r_2)$ and hence the distance between the plants can exceed unity. This will cause further issues when a connection between the distance between plants and the associated performance measure $b_{P,C}$ given by \eqref{eqn_bPC_definition} is made for analysing the probabilistic robustness, as $b_{P,C}$ does not exceed the value of $1$. This does not mean that the geodesic metric $\mathrm{d_{geo}}(r_1, r_2)$ is a wrong distance metric choice. Rather, it just means that the corresponding robustness measure that is similar to $b_{P,C}$ and that can handle distance between plants greater than $1$ is yet to be developed, and hence it is just a limitation due to the missing theory. Therefore, it is preferable to use the chordal distance in the subsequent theoretical development to reflect the normalized distance value in $[0, 1]$, facilitating future developments regarding probabilistic robust performance results using the performance measure $b_{P,C}$. From the optimal transport perspective, the transport plan will happen through the interior of the Riemann sphere (which is perfectly fine) when the transport cost is computed in terms of the chordal distance metric. 
\end{remark}
\subsubsection{Upper Bound on $\mathrm{d}_{q}\left(P_1, P_2\right)$}
Given any support set, it is possible to define an uniform distribution over it. Using this simple observation, the connection between the support distance and the proposed type-$q$ distance between the systems $P_1$ and $P_2$ is established in the following theorem.
\begin{theorem} \label{theorem_support_distance_upper_bound}
For $q \geq 1$, let $\mathbb{P}_{R_{\ell}(\omega)} := \mathcal{U}(\mathcal{R}_{P_{\ell}}(\omega))$ be the uniform distribution defined over the support of the projected uncertainty for each system $\ell = \{1, 2\}$. Then,
\begin{align}
\label{eqn_theorem_support_distance_upper_bound}
\mathrm{d}_{q}(P_1, P_2) 
\leq 
\mathrm{d^{\mathfrak{R}}_{sup}}(P_1, P_2)^{q}.
\end{align}
\end{theorem}

\begin{proof}
Given $q \geq 1$, fix any $ \omega \in \Omega $. Let $\mu := \mathcal{U}(\mathcal{R}_{P_{1}}(\omega))$ and $\nu := \mathcal{U}(\mathcal{R}_{P_{2}}(\omega))$ be the uniform distributions over the respective compact support sets. Then, the chordal metric induced type-$q$ Wasserstein distance using \eqref{eqn_distance_btw_p1_p2} is given by
\begin{align*}
W_q^q(\mu, \nu) = \inf_{\pi \in \Pi(\mu, \nu)} \int_{\mathcal{R}_{P_{1}}(\omega) \times \mathcal{R}_{P_{2}}(\omega)} \! \! \! \! \! \mathrm{d_{chord}}(r_1, r_2)^q \, d\pi(r_1, r_2).
\end{align*}
Given that the supports of $\mu$ and $\nu$ are $\mathcal{R}_{P_{1}}(\omega)$, and $\mathcal{R}_{P_{2}}(\omega)$ respectively, the support distance between $\mathcal{R}_{P_{1}}(\omega)$, and $\mathcal{R}_{P_{2}}(\omega)$ denoted by $\mathrm{d^{\mathfrak{R}}_{sup}}(P_1, P_2, \omega)$ is given by \eqref{eqn_distance_btw_r}. Then, for any joint distribution $\pi \in \Pi(\mu, \nu)$, we see that 
\begin{align*}
\int_{\mathcal{R}_{P_{1}}(\omega) \times \mathcal{R}_{P_{2}}(\omega)} \! \! \mathrm{d_{chord}}(r_1, r_2)^{q} \, d\pi(r_1, r_2) 
\leq 
\mathrm{d^{\mathfrak{R}}_{sup}}(P_1, P_2, \omega)^{q}.
\end{align*}
Taking the infimum over $\pi \in \Pi(\mu, \nu)$ and using the fact that $q \geq 1$, we see that
\begin{align*}
W_q^q(\mu, \nu) 
\leq 
\mathrm{d^{\mathfrak{R}}_{sup}}(P_1, P_2, \omega)^{q}.
\end{align*}
Taking supremum over all $\omega \in \Omega$, we get \eqref{eqn_theorem_support_distance_upper_bound} and the proof is complete.
\end{proof}
\subsubsection{Lower Bound on $\mathrm{d}_{q}\left(P_1, P_2\right)$}
Having obtained an upper bound for the distance metric $\mathrm{d}_{q}(P_1, P_2)$ in  Theorem \ref{theorem_support_distance_upper_bound}, we now proceed below to get a lower bound using triangle inequality based arguments. We will leverage the nominal distance and the expected deviation of the random plant instances of each systems from their respective nominal models to arrive at a lower bound for the proposed distance measure. 
\begin{theorem}
\label{theorem_triangle_lower_bound}
Suppose $\forall \omega \in \Omega$, the nominal frequency response of the systems $P_{\ell}$ for $\ell = \{1, 2\}$ denoted by $\overline{P}_{\ell}(j\omega) \in \mathbb{C}$ be known apriori. Let $\overline{R}_{\ell}(\omega) = \varphi^{-1}(\overline{P}_{\ell}(j\omega))$ denote their corresponding inverse stereographic projections onto the Riemann sphere. Further, let $\mathbb{P}_{R_{\ell}(\omega)}$ be the known probability distribution governing the projected uncertainty of $P_{\ell}(j \omega)$ in the Riemann sphere. Then, 
\begin{align}
\label{eqn_theorem_triangle_lower_bound}    
\mathrm{d}_{q}(P_1, P_2) 
&\geq
\sup_{\omega \in \Omega}
\inf_{\pi_{\omega} \in \Pi_{\omega}}
\Delta_{\mathrm{dev}}(\omega), 
\end{align}
where the deviations from the nominal response at frequency $\omega \in \Omega$ is given by 
\begin{subequations}
\begin{align}
\Delta_{\mathrm{dev}}(\omega)
&=
\left( 
\mathrm{d_{chord}}(\overline{R}_{1}(\omega), \overline{R}_{2}(\omega))\! - \!
\mathbb{E} 
\left[
\Delta_{\mathrm{nom}}(\omega) 
\right]
\right)^{q}_{+} \\
\Delta_{\mathrm{nom}}(\omega)
&= 
\sum^{2}_{\ell = 1}
\mathrm{d_{chord}}(r_{\ell}(\omega), \overline{R}_{\ell}(\omega)) \label{eqn_deltaNom_definition}.
\end{align}
\end{subequations}
\end{theorem}

\begin{proof}
Given $q \geq 1$, fix any frequency $\omega \in \Omega$ and let $\pi_{\omega} \in \Pi_{\omega}$ be any admissible joint distribution at that frequency. For the brevity of notation, we will drop the $\omega$ argument from the quantities of interests when it is obvious that quantities are function of frequency $\omega$. For any $r_1 \in \mathcal{R}_{P_{1}}$ and $r_2 \in \mathcal{R}_{P_{2}}$, we can define $\Delta_{\mathrm{nom}}$ that denotes the sum of deviation (measured in terms of the chordal distance metric) of any perturbed models from their respective nominal models using \eqref{eqn_deltaNom_definition} and it is a random variable. Using the triangle inequality for the $\mathrm{d_{chord}}$ using $\Delta_{\mathrm{nom}}$, we see that
\begin{align*}
&\mathrm{d_{chord}}(r_1, r_2) \geq \left( \mathrm{d_{chord}}(\overline{R}_{1}, \overline{R}_{2}) - \Delta_{\mathrm{nom}} 
\right)_{+} \\
\iff 
&\mathrm{d_{chord}}(r_1, r_2)^q \geq \left( 
\mathrm{d_{chord}}(\overline{R}_{1}, \overline{R}_{2}) - \Delta_{\mathrm{nom}} 
\right)^{q}_{+}.
\end{align*}
Since $\Delta_{\mathrm{nom}}$ is a random variable, we 
take expectation with respect to the joint distribution $\pi_{\omega}$ on both sides to get
\begin{align*}
&\mathbb{E}_{(r_1, r_2) \sim \pi_{\omega}} \left[ \mathrm{d_{chord}}(r_1, r_2)^q \right] \\
\geq
&\mathbb{E}_{(r_1, r_2) \sim \pi_{\omega}} 
\left[
\left( 
\mathrm{d_{chord}}(\overline{R}_{1}, \overline{R}_{2}) - \Delta_{\mathrm{nom}} 
\right)^{q}_{+}
\right] \\
\geq
&\left( 
\mathrm{d_{chord}}(\overline{R}_{1}, \overline{R}_{2}) - 
\mathbb{E}_{(r_1, r_2) \sim \pi_{\omega}} 
\left[
\Delta_{\mathrm{nom}} 
\right]
\right)^{q}_{+}.
\end{align*}
where we applied Jensen's inequality (for convex function $f(x) = (x)^{q}_{+}$ for $q \geq 1$, Jensen's inequality implies $\mathbb{E}[(x)^{q}_{+}] \geq (\mathbb{E}[x])^{q}_{+}$) on the right hand side. Note that $\mathbb{E}_{(r_1, r_2) \sim \pi_{\omega}} \left[ \mathrm{d_{chord}}(r_1, r_2)^q \right]$ is simply equal to $\int_{\mathcal{R}_{P_{1}} \times \mathcal{R}_{P_{2}}} \mathrm{d_{chord}}(r_1, r_2)^q \, d\pi_{\omega}$. Taking infimum over all joint distributions $\pi_{\omega} \in \Pi(\mathbb{P}_{R_{1}(\omega)}, \mathbb{P}_{R_{2}(\omega)})$ on both sides, we get
\begin{align*}
&W_q^q(\mathbb{P}_{R_{1}(\omega)}, \mathbb{P}_{R_{2}(\omega)}) \\
&\geq 
\inf_{\pi_{\omega} \in \Pi_{\omega}}
\left( 
\mathrm{d_{chord}}(\overline{R}_{1}, \overline{R}_{2}) \! - \!
\mathbb{E}_{(r_1, r_2) \sim \pi_{\omega}} 
\left[
\Delta_{\mathrm{nom}} 
\right]
\right)^{q}_{+}.
\end{align*}
Taking supremum over all $\omega \in \Omega$ on both sides yields \eqref{eqn_theorem_triangle_lower_bound} and the proof is complete.
\end{proof}

\begin{remark}
We believe that a similar lower bound like the one in \eqref{eqn_theorem_triangle_lower_bound} for the $\mathrm{d}_{q}(P_1, P_2)$ can be obtained using \cite[Theorem~14.60]{rockafellar1998variational} by adapting the integral-infimum interchange theorem to supremum, where interchange of supremum and integration happens under conditions favouring Fubini-type arguments. Similarly, we can use Dobrushin inequality given any feasible candidate transport plan to obtain a simple yet conservative upper bound for the distance measure. We leave both the expositions as future works.
\end{remark}

\subsection{$\mathbb{P}_{P_{1}(\omega)}, \mathbb{P}_{P_{2}(\omega)}$ Being Empirical Distributions} 
Suppose that $\forall \omega \in \Omega$, for system $\ell \in \{1, 2\}$,
we construct the empirical distribution $\mathbb{P}_{P_{\ell}(\omega)}$ using samples of frequency response data $\left\{ \hat{P}^{(i)}_{\ell}(j \omega) \right\}^{N}_{i=1}$ obtained through $N \in \mathbb{N}$ independent system identification trials. That is, $\forall \omega \in \Omega$, we construct the empirical distribution as
\begin{align}
\label{eqn_empirical_distrbutions}
\mathbb{P}_{P_{\ell}(\omega)} 
&=
\frac{1}{N} \sum^{N}_{i = 1} \boldsymbol{\delta}_{\hat{P}^{(i)}_{\ell}(j \omega)}, \quad \ell = 1,2,
\end{align}
where $\boldsymbol{\delta}_{\hat{P}^{(i)}_{\ell}(j \omega)}$ denotes the Dirac delta measure concentrated at the point $\hat{P}^{(i)}_{\ell}(j \omega) \in \mathbb{C}$. Then, using \eqref{eq:projected_distribution} from Proposition \ref{proposition_projected_distribution}, one can obtain the corresponding projected distribution $\mathbb{P}_{R_{\ell}(\omega)}$, for each system $\ell = 1,2$. The following proposition describes the computation of distance metric for this special case of empirical distributions. 

\begin{proposition}
\label{proposition_empirical_distributions_distance}
Given $q \geq 1$, suppose that the empirical distributions $\mathbb{P}_{P_{1}(\omega)}$ and $\mathbb{P}_{P_{2}(\omega)}$ of both the systems $P_1$ and $P_2$ are given by \eqref{eqn_empirical_distrbutions}. Then, 
\begin{equation}
\label{eqn_empirical_distance_p1_p2}
\footnotesize
\mathrm{d}_{q}(P_1, P_2) 
\!
=
\!
\sup_{\omega \in \Omega}
\inf_{\pi_\omega \in \Pi_{\omega}} \!
\sum^{N}_{i = 1}
\sum^{N}_{k = 1}
\mathrm{d_{chord}}\left(\hat{R}^{(i)}_{1}(j \omega), \hat{R}^{(k)}_{2}(j \omega)\right)^{q} \!\!
\pi_\omega(i,k),
\end{equation}
where $\hat{R}^{(i)}_{\ell}(j \omega) = \phi^{-1}\left(\hat{P}^{(i)}_{\ell}(j \omega) \right)$ for $\ell = 1,2$ and $\pi_\omega(i,k)$ denotes a valid joint distribution between the distributions at frequency $\omega \in \Omega$. 
\end{proposition}
\begin{proof}
Adapting the distance calculation from \eqref{eqn_distance_btw_p1_p2} to the Wasserstein distance between empirical distributions $\mathbb{P}_{R_{1}(\omega)}, \mathbb{P}_{R_{2}(\omega)}$ using the chordal metric given by \eqref{eqn_geodesic_distance_definition} yields the result \eqref{eqn_empirical_distance_p1_p2}.
\end{proof}

The continuity in the frequency $\omega \in \Omega$ dimension still makes the exact computation of \eqref{eqn_empirical_distance_p1_p2} hard. One way to approximately address this problem is to discretize the frequency dimension as finely as possible and deploy numerical approximation methods. That is, we can define the discretized frequency space with $M \in \mathbb{N}$ points as $\Omega_{M} = \left\{ \omega_i \in \Omega \right\}^{M}_{i = 1}$. Then, the distance in \eqref{eqn_empirical_distance_p1_p2} can be approximated as
\begin{equation}
\label{eqn_approx_empirical_distance_p1_p2}
\footnotesize
\hat{\mathrm{d}}_{q}(P_1, P_2)
\!
=
\!
\max_{\omega \in \Omega_{M}} \!
\inf_{\pi_\omega \in \Pi_{\omega}} \!
\sum^{N}_{i = 1}
\sum^{N}_{k = 1}
\mathrm{d_{chord}}(\hat{R}^{(i)}_{1}(j \omega), \hat{R}^{(k)}_{2}(j \omega))^{q}
\pi_\omega(i,k).   
\end{equation}
We will now aim to get an empirical upper bound on $\hat{\mathrm{d}}_{q}(P_1, P_2)$ using similar earlier arguments in the following proposition.
\begin{proposition}
\label{proposition_empirical_Hausdorff_upper_bound}
For each $\omega \in \Omega_M$, let $\mathbb{P}_{R_{\ell}(\omega)}$ for system $\ell = \{1, 2\}$ be its empirical distribution defined over $N$ samples supported on $\partial \mathcal{R}$ and given by
\begin{align}
\label{eqn_empirical_PR_distribution}
\mathbb{P}_{R_{\ell}(\omega)} 
= 
\frac{1}{N} \sum_{i=1}^N \boldsymbol{\delta}_{R_\ell^{(i)}(j\omega)}, \quad \ell = 1,2.
\end{align}
Let $\hat{\pi}_{\omega}$ be any admissible joint distribution in $\Pi(\mathbb{P}_{R_{1}(\omega)}, \mathbb{P}_{R_{2}(\omega)})$ at frequency $\omega$ and the associated worst-case empirical cost at that frequency be defined as 
\begin{subequations}
\label{eqn_empirical_transport_cost}    
\begin{align}
C(\omega)
&:= 
\sup_{i \in [1,N],  
k \in [1,N]} \mathrm{d_{chord}}\left(R_1^{(i)}(j\omega), R_2^{(k)}(j\omega)\right).
\end{align}
\end{subequations}
Then, given $q \geq 1$, we see that
\begin{equation}
\label{eqn_proposition_empirical_Hausdorff_upper_bound}
\hat{d}_q(P_1, P_2)
\leq 
\max_{\omega \in \Omega_M} (C(\omega))^{q}.
\end{equation}
\end{proposition}

\begin{proof}
Fix any frequency $\omega \in \Omega_M$ and let $\hat{\pi}_{\omega}$ be any admissible joint distribution at that frequency. Since $\mathbb{P}_{R_{1}(\omega)}$ and $\mathbb{P}_{R_{2}(\omega)}$ are empirical distributions with $N$ points, (each with equal probability), we see that for any $i = 1, \dots N$ and $k = 1, \dots, N$, 
\begin{align*}
\mathrm{d_{chord}}\left((R_1^{(i)}(j\omega), R_2^{(k)}(j\omega)\right) 
&\leq 
C(\omega) \\
\iff
\mathrm{d_{chord}}\left((R_1^{(i)}(j\omega), R_2^{(k)}(j\omega)\right)^{q} 
&\leq 
(C(\omega))^{q},
\end{align*}
where $C(\omega)$ is given by \eqref{eqn_empirical_transport_cost}. Taking expectation with respect to the joint distribution $\hat{\pi}_{\omega}$ on both sides, we get
\begin{align*}
W_q^q(\mathbb{P}_{R_{1}(\omega)}, \mathbb{P}_{R_{2}(\omega)})
&\leq 
\sum_{i=1}^{N} 
\sum_{k=1}^{N} (C(\omega))^{q} \, \hat{\pi}_{\omega}(i, k) \\
&=
(C(\omega))^{q}
\sum_{i=1}^{N} 
\sum_{k=1}^{N}  \, \hat{\pi}_{\omega}(i, k) \\
&=
(C(\omega))^{q}.
\end{align*}
Taking infimum over the set of all possible joint distributions $\Pi_{\omega}$ and subsequently taking maximum over $\omega \in \Omega_M$ on both sides gives \eqref{eqn_proposition_empirical_Hausdorff_upper_bound}.
\end{proof}

Following the similar thinking of Theorem \ref{theorem_triangle_lower_bound}, we will now obtain a lower bound for the empirical version of the distance measure $\hat{d}_q(P_1, P_2)$ based on the triangle inequality involving the deviation of the random models of each system from their nominal models. 

\begin{theorem}
\label{theorem_triangle_lower_bound_time_domain_empirical}
Suppose at each frequency $\omega \in \Omega_M$, the systems $P_1$ and $P_2$ admit known nominal models $\overline{P}_1(j\omega), \overline{P}_2(j\omega) \in \mathbb{C}$ and let $\overline{R}_{1}(\omega) := \phi^{-1}(\overline{P}_1(j\omega))$ and $\overline{R}_{2}(\omega) := \phi^{-1}(\overline{P}_2(j\omega))$. Further, let the empirical distribution denoted by $\mathbb{P}_{R_{\ell}(\omega)}$ for system $\ell = \{1, 2\}$ charactering its respective system uncertainty at the frequency $\omega \in \Omega_M$ be given by \eqref{eqn_empirical_PR_distribution} and additionally let $\pi_{\omega} \in \Pi(\mathbb{P}_{R_{1}(\omega)}, \mathbb{P}_{R_{2}(\omega)})$ be any feasible joint distribution at the frequency $\omega \in \Omega_M$. Then, for $i, k \in \{i, \dots, N\}$,
\begin{align}
\Delta^{(ik)}_{\mathrm{nom}}(\omega)    
&:= 
\mathrm{d_{chord}}\left(\omega; R_{1}^{(i)}, \overline{R}_{1} \right)
+
\mathrm{d_{chord}}\left(\omega; R_{2}^{(k)}, \overline{R}_{2} \right),  \label{eqn_deviation_from_nominal} \\
\hat{d}_q(P_1, P_2) 
&\geq 
\max_{\omega \in \Omega_M} 
\inf_{\pi_{\omega} \in \Pi_{\omega}}
\left( 
\Delta_{\mathrm{dev}}(\omega)
\right)^{q}_{+}, \label{eqn_theorem_triangle_lower_bound_time_domain_empirical} \\
\Delta_{\mathrm{dev}}(\omega)
&= 
\mathrm{d_{chord}}(\omega; \overline{R}_{1}, \overline{R}_{2})\! - \!
\sum^{N}_{i = 1}
\sum^{N}_{k = 1}
\Delta^{(ik)}_{\mathrm{nom}}(\omega) \, 
\pi_{\omega}(i, k).
\end{align}
\end{theorem}

\begin{proof}
We will follow the arguments in proof of Theorem \ref{theorem_triangle_lower_bound}. Let us fix a frequency $\omega \in \Omega_{M}$. As usual, for the brevity of notation, we will drop the $\omega$ argument when things are obvious. Using the triangle inequality for the chordal distance metric $\mathrm{d_{chord}}$ on the Riemann sphere involving the deviation of system models from their respective nominal models, we see that for every sample $R_1^{(i)} \in \mathrm{supp}(\mathbb{P}_{R_1})$ and $R_2^{(k)} \in \mathrm{supp}(\mathbb{P}_{R_2})$ with $i,k = \{1, \dots, N\}$, we see that 
\begin{align}
\label{eqn_theorem_triangle_lower_bound_time_domain_empirical_interim_1}
\mathrm{d_{chord}}\left(R_1^{(i)}, R_2^{(k)}\right) 
\geq 
\left( \mathrm{d_{chord}}(\overline{R}_{1}, \overline{R}_{2}) - \Delta^{(ik)}_{\mathrm{nom}} \right)_{+},
\end{align}
where the deviation from the respective system's nominal models $\Delta^{(ik)}_{\mathrm{nom}}$ is given by \eqref{eqn_deviation_from_nominal}. Raising \eqref{eqn_theorem_triangle_lower_bound_time_domain_empirical_interim_1} to the power $q \geq 1$, we get
\begin{align*}
\mathrm{d_{chord}}\left(R_1^{(i)}, R_2^{(k)}\right)^{q} 
\geq 
\left( 
\mathrm{d_{chord}}(\overline{R}_{1}, \overline{R}_{2}) - \Delta^{(ik)}_{\mathrm{nom}} \right)^{q}_{+}.
\end{align*}
Taking expectation with respect to the joint distribution $\pi_{\omega}$ of $\mathbb{P}_{R_{\ell}(\omega)}$ for system $\ell = \{1, 2\}$ on both sides, we get
\begin{align*}
&\sum^{N}_{i = 1}
\sum^{N}_{k = 1}
\mathrm{d_{chord}}\left(R_1^{(i)}, R_2^{(k)}\right)^{q} \, \pi_{\omega}(i, k) \\
&\geq 
\sum^{N}_{i = 1}
\sum^{N}_{k = 1}
\left( 
\mathrm{d_{chord}}(\overline{R}_{1}, \overline{R}_{2}) - \Delta^{(ik)}_{\mathrm{nom}} \right)^{q}_{+} \, \pi_{\omega}(i, k) \\
&\geq 
\left( 
\underbrace{
\mathrm{d_{chord}}(\overline{R}_{1}, \overline{R}_{2}) -
\sum^{N}_{i = 1}
\sum^{N}_{k = 1}
\Delta^{(ik)}_{\mathrm{nom}}  \pi_{\omega}(i, k)
}_{:= \Delta_{\mathrm{dev}}(\omega)}
\right)^{q}_{+},
\end{align*}
where we used the Jensen's inequality as we did before in the proof of Theorem \ref{theorem_triangle_lower_bound}. Subsequently, taking the infimum over all possible joint distributions $\pi_{\omega} \in \Pi(\omega; \mathbb{P}_{R_1}, \mathbb{P}_{R_2})$, and maximizing over $\omega \in \Omega_{M}$ on both sides yields \eqref{eqn_theorem_triangle_lower_bound_time_domain_empirical}.
\end{proof}

\subsection{Numerical Demonstrations}
\begin{figure}
    \centering
    \includegraphics[width=\linewidth]{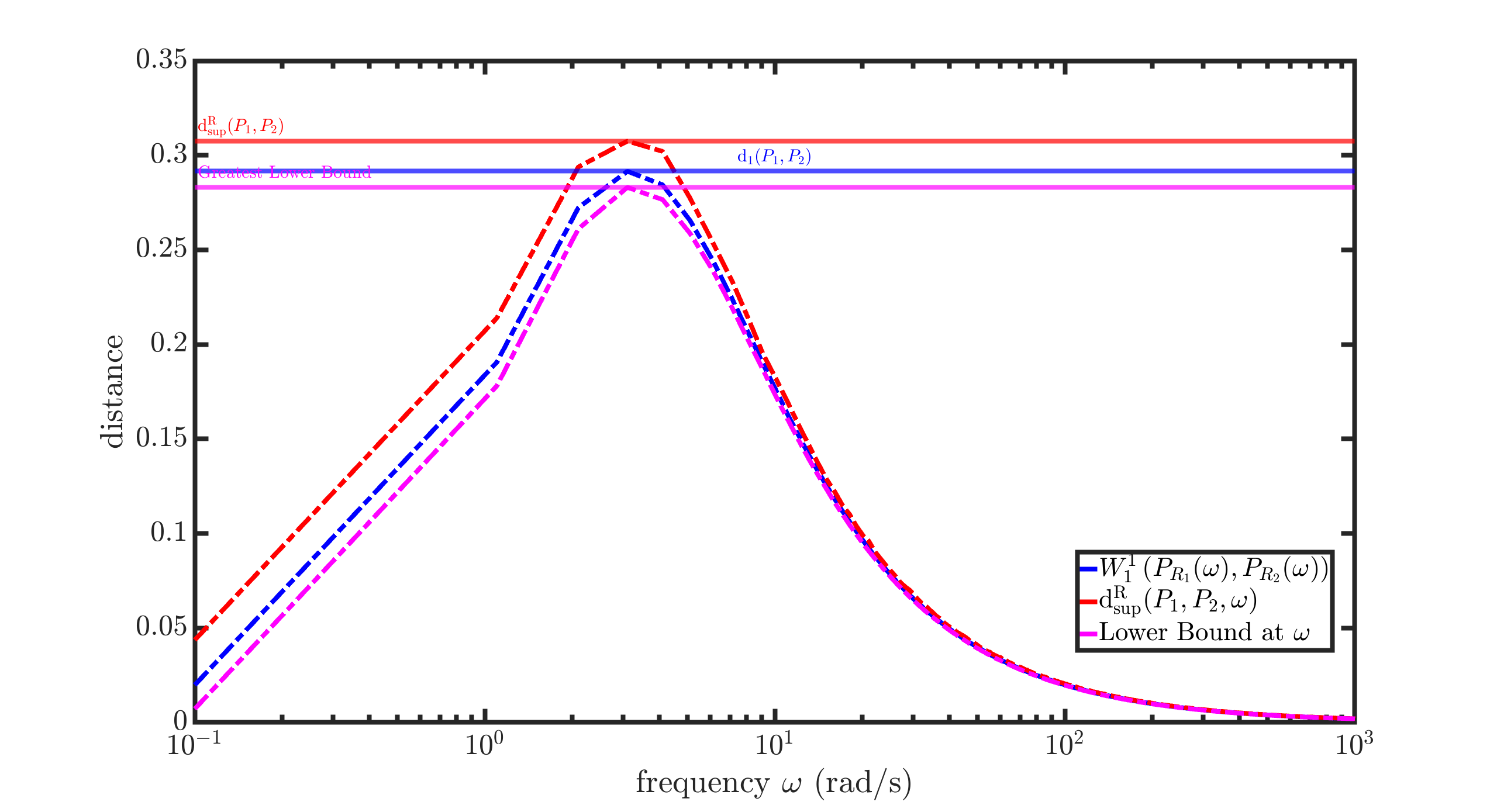}
    \caption{The chordal metric based type-1 Wasserstein distance $\hat{d}_1(P_1, P_2)$ between two systems $P_{1}(s) = \frac{1}{1+0.5s}$ and $P_{2}(s) = \frac{1}{(1+0.2s)(1+0.7s)}$ is shown here in blue colour. The upper bound using the support distance is given in red colour and its lower bound from Theorem \ref{theorem_triangle_lower_bound_time_domain_empirical} is shown in magenta colour. Quantities that are functions of frequency is given by solid lines and their respective maximum values are shown in dashed horizontal lines in the same colour.}
    \label{fig_frequencySimulation}
\end{figure}
To demonstrate the proposed distance in the frequency domain, we consider two different systems with their nominal model transfer functions given by $P_{1}(s) = \frac{1}{1+0.5s}$ and $P_{2}(s) = \frac{1}{(1+0.2s)(1+0.7s)}$ respectively. A discretized frequency space $\Omega_{M}$ containing $M = 1000$ points between $[0.1, 10^{3}] \textrm{ rad}/s$ was formed. At every frequency $\omega \in \Omega_{M}$, $N = 100$ samples of frequency response data were generated by randomly perturbing the nominal frequency response at that frequency. Both the nominal frequency response and the empirical distribution containing the samples at every frequency were projected onto the Riemann sphere using the inverse Stereographic projection given by \eqref{eqn_sphere_cartesian_coords}. The chordal distance metric based type-1 Wasserstein distance $\hat{d}_1(P_1, P_2)$ between two systems given by \eqref{eqn_approx_empirical_distance_p1_p2} was computed using the linear programming approach. The upper bound using the support distance was computed using Proposition \ref{proposition_empirical_Hausdorff_upper_bound} and the lower bound was computed using Theorem \ref{theorem_triangle_lower_bound_time_domain_empirical}. The results are shown in Figure \ref{fig_frequencySimulation}. The proposed frequency domain distance $\hat{d}_1(P_1, P_2)$ between the systems $P_{1}$ and $P_{2}$ was found to be $0.2916$. The corresponding upper and lower bounds computed using Proposition \ref{proposition_empirical_Hausdorff_upper_bound} and Theorem \ref{theorem_triangle_lower_bound_time_domain_empirical} were found to be $0.3075$ and $0.2831$ respectively. 

\subsection{Summary of Frequency Domain Distance}
For SISO systems, the proposed distance measure in the frequency domain using the chordal distance measure mimicking the $\nu$-gap perspective comes in handy with nice pictorial illustration. We believe that this is just a starting point and there are several interesting future research extensions. Though in principle, we expect the theory to carry forward in a similar fashion from SISO to MIMO systems setting, we expect some inherent difficult that comes with higher dimensions to kick in. For example, we would be required to work with the Riemann sphere of higher dimensions and associated stereographic projections are more mathematically involved and complex in nature. Another problem is of dealing with the supremum with respect to the frequency parameter. This problem persists even in SISO and will continue to persist even in MIMO setting. To get around these issues, we would like to formulate and obtain an analogous distance measure between stochastic linear systems in the time domain setting using gap metric perspective where the process of taking supremum with respect to $\omega \in \Omega$ would be absent. The exposition with the time domain setting will be carried out in the next section.

\section{Problem Formulation in Time Domain} \label{sec:time_domain}
In this section, we will present an analogous distance measure in the time domain using the gap metric. 

\subsection{Uncertain Dynamical Systems}
Consider two continuous time LTI dynamical systems living in the space of linear stochastic systems denoted by $\mathbf{\Sigma}$. Let the nominal models of both systems $i \in \{1, 2\}$ be given by
\begin{align}
\label{eqn_nominal_dynamics_system}
\bar{\Sigma}_{i}
:
\Bigl\{
\dot{x}_{i}(t) 
=
A_{i} x_{i}(t) + B_{i} u_{i}(t), \quad y_{i}(t) = C_{i} x(t), 
\end{align}
where for system $i$ at time $t \in \mathbb{R}_{\geq 0}$, we refer to its system states as $x_{i}(t) \in \mathbb{R}^{n}$, the control inputs to the system as $u_{i}(t) \in \mathbb{R}^{m}$, the system outputs as $y_{i}(t) \in \mathbb{R}^{l}$ and the matrices $A_{i} \in \mathbb{R}^{n \times n}$, $B_{i} \in \mathbb{R}^{n \times m}$, $C_{i} \in \mathbb{R}^{l \times n}$. Real-world dynamical systems usually have some form of uncertainties associated with them either due to the lack of modelling tools or due to the inaccuracies of the modelling framework. Hence, in practice, all systems have inherent uncertainties affecting their evolution. We model the uncertainty affecting the evolution of the uncertain system $i \in \{1, 2\}$ using the parameter $\theta_{i} \in \mathbb{R}^{p}$ with $p \leq (n^{2}+nm+ln)$ and $\theta$ directly affects the evolution of the perturbed system described as follows:
\begin{equation}
\label{eqn_perturbed_dynamics_system}
\Sigma_{i}(\theta_{i})
:
\left\{ \begin{aligned} 
  \dot{x}_{i}(t) 
&= 
A_{i}(\theta_{i}) x_{i}(t) + B_{i}(\theta_{i}) u_{i}(t), \\
y_{i}(t) 
&= C_{i}(\theta_{i}) x_{i}(t).    
\end{aligned} \right.
\end{equation} 
The matrices of the perturbed model of the $i$\textsuperscript{th} system given by $(A_{i}(\theta_{i}), B_{i}(\theta_{i}), C_{i}(\theta_{i}))$ are of the same dimensions as $(A_{i}, B_{i}, C_{i})$ respectively. We will assume that $\theta_{i} \sim \mathbf{f_{\theta_{i}}}$ where, $\mathbf{f_{\theta_{i}}}$ denotes the distribution of the parameter $\theta_{i}$ affecting the evolution of system $i \in \{1, 2\}$. It is possible to take a moment-based ambiguity set formulation for $\mathbf{f_{\theta_{i}}}$ by assuming that $\mathbf{f_{\theta_{i}}}$ is unknown but is believed to be belonging to a moment-based ambiguity set denoted by $\mathcal{P}^{\theta_{i}}$ consistent with mean $\mu_{\theta_{i}} \in \mathbb{R}^{p}$ and covariance $\Sigma_{\theta_{i}} \succ 0$. However, for the ease of exposition, we will assume that $\theta_{i} \sim \mathbf{f}_{\theta_{i}} = \mathcal{N}(\mu_{\theta_{i}}, \Sigma_{\theta_{i}})$. We note here that $\Sigma_{i}(\bar{\theta}_{i}) = \bar{\Sigma}_{i}$ meaning that the perturbed system equals the nominal system when uncertainty vanishes at $\bar{\theta}_{i}$ for system $i$. This does not imply that $\mu_{\theta_{i}} = \bar{\theta}_{i}$. The only requirement that is needed is that $\bar{\theta}_{i} \in \mathbf{f}_{\theta_{i}}$ (perfectly fine even if the containment happens asymptotically (as number of samples tend to $\infty$)) so that when the uncertainties of the perturbed system vanish, it results in the nominal system.

\subsection{Gap Between Models}
Having defined the evolution of the nominal model $\bar{\Sigma}_{i}$ using \eqref{eqn_nominal_dynamics_system} and perturbed model $\Sigma_{i}(\theta_{i})$ in \eqref{eqn_perturbed_dynamics_system} for both the systems $i \in \{1, 2\}$, the gap between the nominal model and the perturbed model of the $i$\textsuperscript{th} system denoted by $\mathrm{Gap}_{i}(\theta_{i})$ for $\theta_{i} \sim \mathbf{f}_{\theta_{i}} = \mathcal{N}(\mu_{\theta_{i}}, \sigma^{2}_{\theta_{i}} I_{p})$ can be defined using \eqref{eqn_gap_metric_definition} as
\begin{align}
\label{eqn_gap_between_models}    
\mathrm{Gap}_{i}(\theta_{i}) 
:=
\delta_{g}(\bar{\Sigma}_{i}, \Sigma_{i}(\theta_{i})).
\end{align}
Clearly, $\mathrm{Gap}_{i}(\theta_{i})$ is a random variable in $(0,1)$ if the $i$\textsuperscript{th} system is stable for all possible perturbations due to $\theta_{i} \sim \mathbf{f}_{\theta_{i}} = \mathcal{N}(\mu_{\theta_{i}}, \Sigma_{\theta_{i}})$. However, we need distance between two systems and rather not between two models of the same system. Towards that we define the distance between the nominal models of two systems $i \in \{1, 2\}$ denoted by $\mathrm{dist^{nom}_{\Sigma_{1}, \Sigma_{2}}}$ as
\begin{align}
\label{eqn_distance_between_nominal_models}        
\mathrm{dist^{nom}_{\Sigma_{1}, \Sigma_{2}}}
:=
\delta_{g}(\bar{\Sigma}_{1}, \bar{\Sigma}_{2}).
\end{align}
That is, when the uncertainties of both systems $i \in \{1, 2\}$ vanish, then it simply boils down to the simple gap metric between two deterministic nominal system models $\bar{\Sigma}_{1}$ and $\bar{\Sigma}_{2}$. However, systems always come with uncertainties due to inevitable modelling errors and hence $\mathrm{dist^{nom}_{\Sigma_{1}, \Sigma_{2}}}$ will not truly capture the distance between the two stochastic systems strictly speaking. 

\subsection{Inferring Distribution of Perturbed Plant}
In this regard, we propose to measure the distance between the distributions that are governing the randomness of the plant models of system $i \in \{1, 2\}$. That is, the randomness in $\theta_{i}$ manifests itself as the randomness in the plant $\Sigma_{i}(\theta_{i})$ meaning that $\Sigma_{i}(\theta_{i}) \sim \mathbf{f}_{i}$, where $\mathbf{f}_{i}$ is the distribution of plant models of $i$\textsuperscript{th} system. Towards this, we define the measurable map from the parameter space $\Theta_{i} \subseteq \mathbb{R}^{p}$ to the space of system plants denoted by $\Phi_{i}: \Theta_{i} \mapsto \Sigma_{i}(\theta_{i}) = (A_{i}(\theta_{i}), B_{i}(\theta_{i}), C_{i}(\theta_{i}))$. Then the distribution $\mathbf{f}_{i}$ of the perturbed model of system $i \in \{1, 2\}$ is the push-forward measure of the distribution of $\theta_{i}$ under the map $\Phi_{i}$. That is,
\begin{align}
\label{eqn_distribution_system_models}
\mathbf{f}_{i} = \left(\Phi_{i}\right)_{\#} \mathbf{f}_{\theta_{i}} = \left(\Phi_{i}\right)_{\#} \mathcal{N}(\mu_{\theta_{i}}, \Sigma_{\theta_{i}}),
\end{align}
where for any Borel measurable set $\mathbf{B} \subseteq \mathbb{R}^{(n^{2}+nm+ln)}$ (the space of system plants), the push-forward measure satisfies
\begin{align}
\label{eqn_pushforward_measure}
 \left(\Phi_{i}\right)_{\#} \mathbf{f}_{\theta_{i}}(\mathbf{B}) 
 =
 \mathbb{P}(\Sigma_{i}(\theta_{i}) \in \mathbf{B})
 =
 \mathbf{f}_{\theta_{i}}(\Phi^{-1}_{i}(\mathbf{B})).
\end{align}

Earlier, we mentioned that $\mathrm{Gap}_{i}(\theta_{i})$ is a random variable in $(0,1)$ assuming that $i$\textsuperscript{th} system is stable for all possible perturbations due to $\theta_{i} \sim \mathbf{f}_{\theta_{i}} = \mathcal{N}(\mu_{\theta_{i}}, \Sigma_{\theta_{i}})$. We borrow the following assumptions and the Lipschitz continuity of $\mathrm{Gap}_{i}(\theta_{i})$ from \cite{RenganathanProbabilisticGap}. 
\begin{assumption} \label{assume_theta_continuity}
For system $i \in \{1, 2\}$, the mapping $\theta_{i} \mapsto G_{i}(\theta_{i}) := \begin{bmatrix} N_{i}(\theta_{i}) \\ D_{i}(\theta_{i}) \end{bmatrix}$ is Fr\'echet differentiable and so RCFs $N_{i}(\theta_{i}), D_{i}(\theta_{i})$ are continuously differentiable in $\theta_{i}$ in the sense of $\mathcal{H}_{\infty}$ norm.
\end{assumption}

The randomness in the parameter $\theta_{i}$ induces randomness in the coprime factors $[N_{i}(\theta_{i}), D_{i}(\theta_{i})]$. Subsequently, the randomness in the coprime factors $[N_{i}(\theta_{i}), D_{i}(\theta_{i})]$ manifests itself as variations in the graph $G_{i}(\theta_{i})$, and hence in the angle between the graph subspaces $\mathcal{G}_{\bar{\Sigma_{i}}}$ and $\mathcal{G}_{\Sigma_{i}(\theta_{i})}$ which then finally leads to the randomness in the associated gap $\mathrm{Gap}_{i}(\theta_{i})$. Despite $\mathrm{Gap}_{i}(\theta_{i})$ being random for each of the system, we want to formulate and obtain a deterministic distance measure between two stochastic linear systems, which by the way is the main motive of this manuscript. Consider the special case when the $\theta_{i}$ parameter dependence on the perturbed system dynamics in \eqref{eqn_perturbed_dynamics_system} is affine with $\mathbf{f}_{\theta_{i}}$ being Gaussian. Then, the distribution $\mathbf{f}_{i}$ of the $i$\textsuperscript{th} system due to \eqref{eqn_distribution_system_models} turns out to be Gaussian as well due to the affine transformation properties of Gaussian random vectors. The following lemma formally establishes this result.

\begin{lemma} \label{lemma_affine_parametric_plant_distribution}
Let $ \theta \sim \mathcal{N}(\mu_{\theta}, \Sigma_{\theta}) $ denote a random parameter vector in $\mathbb{R}^{p}$. Given $d \in \mathbb{N}$, consider the state-space matrices of the perturbed LTI dynamical system $\Sigma(\theta)$ whose affine dependence on the parameter $\theta$ is given by
\begin{subequations}
\label{eqn_lemma_affine_parametric_plant_matrices}
\begin{align}
A(\theta) 
&= 
A^0 + \sum_{k=1}^d \theta^{(k)} A^{(k)}, \\
B(\theta) 
&= 
B^0 + \sum_{k=1}^d \theta^{(k)} B^{(k)}, \\
C(\theta) 
&= 
C^0 + \sum_{k=1}^d \theta^{(k)} C^{(k)},
\end{align}
\end{subequations}
where each of the $A^{(k)} \in \mathbb{R}^{n \times n}$, $B^{(k)} \in \mathbb{R}^{n \times m}$, and $C^{(k)} \in \mathbb{R}^{l \times n}$ for $k = 1, \dots, d$ are known apriori. Let $z := \mathrm{vec}(A(\theta), B(\theta), C(\theta)) \in \mathbb{R}^{p}$ where $p = n^2 + nm + ln$. Then, 
\begin{align}
\label{eqn_z_distribution}
z \sim \mathcal{N}(J \mu_{\theta} + z^{0},\, J \Sigma_{\theta} J^{\top}),    
\end{align}
where, $z^{0} = 
\begin{bmatrix}
\mathrm{vec}(A^0) \\
\mathrm{vec}(B^0) \\
\mathrm{vec}(C^0)
\end{bmatrix} \in \mathbb{R}^p$,
and 
$J = 
\begin{bmatrix}
\mathrm{vec}(A^{(1)}) & \cdots & \mathrm{vec}(A^{(d)}) \\
\mathrm{vec}(B^{(1)}) & \cdots & \mathrm{vec}(B^{(d)}) \\
\mathrm{vec}(C^{(1)}) & \cdots & \mathrm{vec}(C^{(d)})
\end{bmatrix} \in \mathbb{R}^{p \times d}$.

\end{lemma}

\begin{proof}
From \eqref{eqn_lemma_affine_parametric_plant_matrices}, we see that
\begin{align*}
\mathrm{vec}(A(\theta)) 
&= \mathrm{vec}\left(A^0 + \sum_{k=1}^d \theta^{(k)} A^{(k)}\right) \\
&= \mathrm{vec}(A^0) + \sum_{k=1}^d \theta^{(k)} \mathrm{vec}(A^{(k)}) \\
\mathrm{vec}(B(\theta)) 
&= \mathrm{vec}\left(B^0 + \sum_{k=1}^d \theta^{(k)} B^{(k)}\right) \\
&= \mathrm{vec}(B^0) + \sum_{k=1}^d \theta^{(k)} \mathrm{vec}(B^{(k)}) \\
\mathrm{vec}(C(\theta)) 
&= \mathrm{vec}\left(C^0 + \sum_{k=1}^d \theta^{(k)} C^{(k)}\right) \\
&= \mathrm{vec}(C^0) + \sum_{k=1}^d \theta^{(k)} \mathrm{vec}(C^{(k)})
\end{align*}
Then, the vectorized perturbed plant will evolve as
\begin{align}
\label{eqn_vectorized_perturbed_plant_evolution_dynamics}
z = z^{0} + \sum_{k=1}^d \theta^{(k)} z^{(k)},
\quad \text{where} \quad
z^{(k)} 
:= 
\begin{bmatrix}
\mathrm{vec}(A^{(k)}) \\
\mathrm{vec}(B^{(k)}) \\
\mathrm{vec}(C^{(k)})
\end{bmatrix}.
\end{align}
Now define $J:= \begin{bmatrix} z^{(1)} & \cdots & z^{(d)} \end{bmatrix} \in \mathbb{R}^{p \times d}$. Using $J$ in \eqref{eqn_vectorized_perturbed_plant_evolution_dynamics}, we get
\begin{align*}
z = z^{0} + J \theta.
\end{align*}
With $\theta \sim \mathcal{N}(\mu_{\theta}, \Sigma_{\theta}) $ and since Gaussianity is preserved under affine transformations, we infer that $z \sim \mathcal{N}(J \mu_{\theta} + z^{0},\; J \Sigma_{\theta} J^{\top})$. Note that result also follows from \eqref{eqn_distribution_system_models}.
\end{proof}

\subsection{Gap Metric Induced Type-$q$ Wasserstein Distance}
Having studied the transformation of the distribution of the parameter under the mapping of the perturbed dynamics to result in the distribution for the perturbed plant models, we are now ready to define the distance between perturbed models of two systems. 
\begin{definition} \label{define_distance_between_systems_time_domain}
The type-$q$ inter-system distance or the gap metric induced type-$q$ Wasserstein distance denoted by $\mathrm{dist^{q}_{\Sigma_{1}, \Sigma_{2}, \delta_{g}}}: \mathbf{f}_{1} \times \mathbf{f}_{2} \rightarrow [0,1]$ between the distributions $\mathbf{f}_{1}, \mathbf{f}_{2}$ governing the randomness of the perturbed models of systems $i \in \{1, 2\}$ supported on $\mathrm{supp}(\mathbf{f}_{1}), \mathrm{supp}(\mathbf{f}_{2}) \subset \mathbf{\Sigma}$ respectively is defined as
\begin{subequations}
\label{eqn_distance_between_distributions_system_models}
\begin{align}
&\mathrm{dist^{q}_{\Sigma_{1}, \Sigma_{2}, \delta_{g}}} \nonumber \\
&:= 
\underset{\pi_{\mathbf{f}} \in \Pi\left(\mathbf{f}_{1}, \mathbf{f}_{2}\right)}{\inf}
\mathbb{E}_{\pi_{\mathbf{f}}} 
\left[ 
\delta_{g}(\Sigma_{1}, \Sigma_{2})^{q}
\right] \\
&=
\underset{\pi_{\mathbf{f}} \in \Pi\left(\mathbf{f}_{1}, \mathbf{f}_{2}\right)}{\inf}  \int_{\mathrm{supp}(\mathbf{f}_{1}) \times \mathrm{supp}(\mathbf{f}_{2})} 
\delta_{g}(\Sigma_{1}, \Sigma_{2})^{q} \, 
\pi_{\mathbf{f}}(d \Sigma_{1}, d \Sigma_{2}) \\
&=
\inf_{\pi_{\mathbf{P}}\in\Pi(\mathbf{f}_{P_{1}},\mathbf{f}_{P_{2}})} \!  
\int_{\mathrm{supp}\left(\mathbf{f}_{P_{1}}\right) \times \mathrm{supp}\left(\mathbf{f}_{P_{2}}\right)} \!\!\! \delta_{g}(P_1,P_2)^{q}\, \pi_{\mathbf{P}}(d P_1, dP_2).
\end{align}    
\end{subequations}
\end{definition}
Note that in \eqref{eqn_distance_between_distributions_system_models}, the term $\Pi\left(\mathbf{f}_{1}, \mathbf{f}_{2}\right)$ refers to the set of all joint distributions between the distributions $\mathbf{f}_{1}, \mathbf{f}_{2}$ in the state space and $\Pi(\mathbf{f}_{P_{1}},\mathbf{f}_{P_{2}})$ denotes the set of all joint distributions in the corresponding transfer function space $\mathbf{R}\mathcal{L}_{\infty}$ with $P_{i} = \mathbf{TF}(\Sigma_{i})$ for $i \in \{1, 2\}$. Then, a joint distribution $\pi_{\mathbf{P}} \in \Pi(\mathbf{f}_{P_{1}},\mathbf{f}_{P_{2}})$ is related to the joint distribution $\pi_{\mathbf{f}} \in \Pi\left(\mathbf{f}_{1}, \mathbf{f}_{2}\right)$ as follows
\begin{align}
\label{eqn_joint_dist_state_space_tf_space_relation}
\pi_{\mathbf{P}} 
= 
(\mathbf{TF}\times \mathbf{TF})_{\#} \, \pi_{\mathbf{f}}.
\end{align}
In this considered setting, the plant models of both the systems are sampled from the distributions $\mathbf{f}_{1}, \mathbf{f}_{2}$ and hence in \eqref{eqn_distance_between_distributions_system_models}, we have used the gap metric to compute the distance between the events (plant models) while evaluating the type-$q$ Wasserstein distance to find $\mathrm{dist^{q}_{\Sigma_{1}, \Sigma_{2}, \delta_{g}}}$. 

\textbf{Remarks:} It is compelling to think that for the special case when both $\mathbf{f}_{1}$ and $\mathbf{f}_{2}$ are Gaussian with with mean $\mu_{i} = J_{i} \mu_{\theta_{i}} + z^{0}_{i}$ and covariance $\Sigma_{i} = J_{i} \Sigma_{\theta_{i}} J^{\top}_{i}$ respectively for system $i \in \{1, 2\}$, we would then have an explicit analytic solution which can readily give the distance between the two systems for $q = 2$. However, it turns out to be not true. Recall that, we are sampling plant models for both the systems from their respective distributions $\mathbf{f}_{1}$ and $\mathbf{f}_{2}$. The explicit Gaussian formula for type-$2$ Wasserstein distance applies only when the cost function is quadratic and convex in Euclidean space (like $c(z_{1}, z_{2}) = \norm{z_{1} - z_{2}}^{2}_{2}$), and the distributions are supported on Euclidean vector spaces with standard geometry. However, in our setting with the plant distributions $\mathbf{f}_{1}$ and $\mathbf{f}_{2}$, the \emph{type-$2$ inter-system distance} given by \eqref{eqn_distance_between_distributions_system_models} with $q = 2$ uses the gap metric as the cost function of the transport plan and thus rendering the Wasserstein distance becoming both distribution-dependent and cost-function-dependent. Hence, the optimal transport plan no longer equals the one derived from linear Gaussian maps. Moreover, the cost function of the transport plan described by the gap metric is non-Euclidean, nonlinear, and defined on equivalence classes of transfer functions (up to coprime factorizations). Its geometry is intrinsically nonlinear, and thus does not facilitate a readily available closed-form expression. Though it does not admit a closed-form expression, numerical approximations through sample-based optimal transport methods can be possible. To get an upper bound, we will leverage a well-known fact from the optimal transport theory that the type-$q$ Wasserstein distance is upper bounded by the diameter of the support set, raised to the power $q$. The following proposition formally establishes that observation to get an upper bound. 

\begin{proposition} \label{proposition_gap_diameter_upper_bound}
Given uniform distributions $\mathbf{f}_{1}, \mathbf{f}_{2}$ with their support given by $\mathrm{supp}(\mathbf{f}_{1})$ and $\mathrm{supp}(\mathbf{f}_{2})$ respectively governing the plant models of systems $\Sigma_{1}, \Sigma_{2}$ respectively, we see that $\forall q \geq 1$,
\begin{align}
\label{eqn_proposition_gap_diameter_upper_bound}
\mathsf{dist}^{q}_{\Sigma_{1}, \Sigma_{2}, \delta_{g}}
\leq 
\inf_{\pi_{\mathbf{f}} \in \Pi(\mathbf{f}_{1}, \mathbf{f}_{2})}
\left(\sup_{ \substack{\Sigma_{1} \in \mathrm{supp}(\mathbf{f}_{1}) \\ \Sigma_{2} \in \mathrm{supp}(\mathbf{f}_{2})}} \delta_{g}(\Sigma_{1}, \Sigma_{2})\right)^{q}.
\end{align}
\end{proposition}
\begin{proof}
For any $\Sigma_{1} \in \mathrm{supp}(\mathbf{f}_{1}), \Sigma_{2} \in \mathrm{supp}(\mathbf{f}_{2})$,
\begin{align*}
\delta_{g}(\Sigma_{1}, \Sigma_{2})
\leq 
\sup_{ \substack{\Sigma_{1} \in \mathrm{supp}(\mathbf{f}_{1}) \\ \Sigma_{2} \in \mathrm{supp}(\mathbf{f}_{2})}} \delta_{g}(\Sigma_{1}, \Sigma_{2}) := \bar{\delta}.
\end{align*}
Given $q \geq 1$, raising both sides to the power $q$, we get
\begin{align*}
\delta_{g}(\Sigma_{1}, \Sigma_{2})^{q} 
&\leq 
\bar{\delta}^{q}.
\end{align*}
Then,
\begin{align*}
&\int_{\mathrm{supp}(\mathbf{f}_{1}) \times \mathrm{supp}(\mathbf{f}_{2})}
\delta_{g}(\Sigma_{1}, \Sigma_{2})^{q} \pi_{\mathbf{f}}(d \Sigma_{1}, d \Sigma_{2}) \\
&\leq
\int_{\mathrm{supp}(\mathbf{f}_{1}) \times \mathrm{supp}(\mathbf{f}_{2})}
\bar{\delta}^{q} \, \pi_{\mathbf{f}}(d \Sigma_{1}, d \Sigma_{2}) \\
&=
\bar{\delta}^{q} \, 
\underbrace{\int_{\mathrm{supp}(\mathbf{f}_{1}) \times \mathrm{supp}(\mathbf{f}_{2})}
\pi_{\mathbf{f}}(d \Sigma_{1}, d \Sigma_{2})}_{= 1} \\
&=
\bar{\delta}^{q}.
\end{align*}
Taking the infimum over all joint distributions $\pi_{\mathbf{f}} \in \Pi(\mathbf{f}_{1}, \mathbf{f}_{2})$ gives \eqref{eqn_proposition_gap_diameter_upper_bound}.
\end{proof}
We provide another upper bound for the proposed distance measure using $\mathrm{dist^{nom}_{\Sigma_{1}, \Sigma_{2}}}$, and the spread of $\theta_{i}$ for $i \in \{1, 2\}$ in the following theorem.
\begin{theorem}
\label{theorem_moment_upper_bound}
Let $\theta_{i} \sim \mathbf{f}_{\theta_{i}}$ for $i \in \{1, 2\}$ and suppose that the distribution $\mathbf{f}_{i}$ of the perturbed model of system $i \in \{1, 2\}$ be given by \eqref{eqn_distribution_system_models}. Further, let assumption \ref{assume_theta_continuity} hold true for both the systems. Assume that $\exists \mathbf{L}_{i} > 0$ such that the gap metric satisfies
\begin{align}
\label{eqn_gap_Lipschitz_assume}
\delta_{g}
\left(
\Phi_{i}(\theta_{i}), \Phi_{i}(\theta^{\prime}_{i})
\right)
\leq
\mathbf{L}_{i} 
\norm{\theta_{i} -\theta^{\prime}_{i}}
\quad \forall \theta_{i}, \theta^{\prime}_{i} \sim \mathbf{f}_{\theta_{i}}.    
\end{align}
Then, $\forall q \geq 1$, we see that
\begin{align}
\label{eqn_theorem_moment_upper_bound}
\mathrm{dist^{q}_{\Sigma_{1}, \Sigma_{2}, \delta_{g}}}
\leq
\left(
\mathrm{dist^{nom}_{\Sigma_{1}, \Sigma_{2}}}
+
\sum^{2}_{i = 1}
\mathbf{L}_{i}
\left(\mathbb{E}
\left[
\norm{\theta_i - \bar\theta_i}^{q}
\right]
\right)^{\frac{1}{q}}
\right)^{q}.
\end{align}
\end{theorem}
\begin{proof}
We know that $\Phi_{i}(\bar{\theta}_{i}) = \Sigma_{i}(\bar{\theta}_{i}) = \bar{\Sigma}_{i}$ for both systems $\Sigma_{i}$ with $i \in \{1, 2\}$. Apply the triangle inequality associated with the gap metric to see
\begin{align*}
\delta_{g} 
(\Phi_{1}(\theta_{1}), \Phi_{2}(\theta_{2}))
&\leq 
\delta_{g}(\bar{\Sigma}_{1}, \bar{\Sigma}_{2})
+ 
\sum^{2}_{i = 1}
\delta_{g}(\Phi_{i}(\theta_{i}), \bar{\Sigma}_{i}) \\
&=
\mathrm{dist^{nom}_{\Sigma_{1}, \Sigma_{2}}} 
+
\sum^{2}_{i = 1}
\delta_{g}(\Phi_{i}(\theta_{i}), \Phi_{i}(\bar{\theta}_{i})).
\end{align*}
Applying the Lipschitz bound in  \eqref{eqn_gap_Lipschitz_assume}, we see that
\begin{align}
\label{eqn_theorem_moment_upper_bound_interim_step_1}
\delta_{g} 
(\Phi_{1}(\theta_{1}), \Phi_{2}(\theta_{2}))
&\leq 
\mathrm{dist^{nom}_{\Sigma_{1}, \Sigma_{2}}}  
+ 
\sum^{2}_{i = 1}
\mathbf{L}_{i} 
\norm{\theta_{i} - \bar{\theta}_{i}}.
\end{align}
Let us define $Z := \delta_{g}\left(\Phi_{1}(\theta_{1}), \Phi_{2}(\theta_{2})\right)$, and 
$U_{i} := \norm{\theta_{i}-\bar{\theta}_{i}}$. Let $\pi_{\theta} \in \Pi(\mathbf{f}_{\theta_{1}}, \mathbf{f}_{\theta_{2}})$ be any joint distribution of $(\theta_{1},\theta_{2})$. Then, the corresponding joint distribution of the systems $\pi_{\mathbf{f}} \in \Pi(\mathbf{f}_{1}, \mathbf{f}_{2})$ and $\pi_{\theta} \in \Pi(\mathbf{f}_{\theta_{1}}, \mathbf{f}_{\theta_{2}})$ are related as
\begin{align}
\label{eqn_joint_distributions_relation}
\pi_{\mathbf{f}} = (\Phi_{1} \times \Phi_{2})_{\#} \pi_{\theta}.     
\end{align}
As a result of \eqref{eqn_joint_distributions_relation}, we see that
\begin{align}
\label{eqn_Lq_norm_equivalents}
&\int_{\mathrm{supp}(\mathbf{f}_{1}) \times \mathrm{supp}(\mathbf{f}_{2})}
\delta_{g}(\Sigma_{1}, \Sigma_{2})^{q} \pi_{\mathbf{f}}(d \Sigma_{1}, d \Sigma_{2}) \nonumber \\
=
&\int_{\mathrm{supp}(\mathbf{f}_{\theta_{1}}) \times \mathrm{supp}(\mathbf{f}_{\theta_{2}})}
Z^{q} \, \pi_{\theta}(d \theta_{1}, d \theta_{2}).
\end{align}
Having said that, taking $\mathcal{L}_{q}$ norm of \eqref{eqn_theorem_moment_upper_bound_interim_step_1} on both sides with respect to the joint distribution $\pi_{\theta}$ and applying Minkowski’s inequality along with the fact that  $\mathrm{dist^{nom}_{\Sigma_{1}, \Sigma_{2}}}$ is constant, its $\mathcal{L}_{q}$ norm is $\norm{\mathrm{dist^{nom}_{\Sigma_{1}, \Sigma_{2}}}}_{\mathcal{L}_{q}} = \mathrm{dist^{nom}_{\Sigma_{1}, \Sigma_{2}}}$, we get pointwise,
\begin{align}
\label{eqn_theorem_moment_upper_bound_interim_step_2}
\norm{Z}_{\mathcal{L}_{q}}
\leq \mathrm{dist^{nom}_{\Sigma_{1}, \Sigma_{2}}} 
+ 
\sum^{2}_{i = 1}
\mathbf{L}_{i} \norm{U_{i}}_{\mathcal{L}_{q}}.
\end{align}
Using the definition of $\mathcal{L}_{q}$ norm, we see that
\begin{align}
\label{eqn_theorem_moment_upper_bound_interim_step_3}
\norm{U_{i}}_{\mathcal{L}_{q}}
=
\left(
\mathbb{E}_{\pi_{\theta}}[U_{i}^{q}]
\right)^{1/q}
=
\left(
\mathbb{E}_{\mathbf{f}_{\theta_{i}}}
\left[
\norm{\theta_{i}-\bar{\theta}_{i}}^q
\right]
\right)^{1/q}.    
\end{align}
Using \eqref{eqn_joint_distributions_relation} and 
\eqref{eqn_theorem_moment_upper_bound_interim_step_3} in \eqref{eqn_theorem_moment_upper_bound_interim_step_2}, we get
\begin{align*}
\left(\mathbb{E}_{\pi_{\mathbf{f}}}[Z^q]\right)^{1/q} \le \mathrm{dist^{nom}_{\Sigma_{1}, \Sigma_{2}}} 
+ 
\sum^{2}_{i = 1}
\mathbf{L}_{i} \left(
\mathbb{E}_{\mathbf{f}_{\theta_{i}}}
\left[
\norm{\theta_{i}-\bar{\theta}_{i}}^q
\right]
\right)^{1/q}.
\end{align*}
Taking the infimum over all joint distributions $\pi_{\mathbf{f}} \in \Pi(\mathbf{f}_{1}, \mathbf{f}_{2})$ and subsequently using \eqref{eqn_Lq_norm_equivalents} and the definition \eqref{eqn_distance_between_distributions_system_models}, we get
\begin{align}
\label{eqn_theorem_moment_upper_bound_interim_step_4}
\left(\mathrm{dist^{q}_{\Sigma_{1}, \Sigma_{2}, \delta_{g}}}\right)^{\frac{1}{q}}
\leq \mathrm{dist^{nom}_{\Sigma_{1}, \Sigma_{2}}} 
+ 
\sum^{2}_{i = 1}
\mathbf{L}_{i} \left(
\mathbb{E}_{\mathbf{f}_{\theta_{i}}}
\left[
\norm{\theta_{i}-\bar{\theta}_{i}}^q
\right]
\right)^{1/q}.
\end{align}
Since $x\mapsto x^q$ is increasing on $\mathbb{R}_{\geq 0}$ for $q \geq 1$, we raise both sides of \eqref{eqn_theorem_moment_upper_bound_interim_step_4} to the power $q \geq 1$ to get \eqref{eqn_theorem_moment_upper_bound} and the proof is complete.
\end{proof}

In the following proposition, we will give a lower bound using the nominal gap distance and the deviation from nominal model. The reasoning will be very similar to that of Theorem  \ref{theorem_triangle_lower_bound} in the frequency domain meaning that we have an analogous result in the time domain given by the following proposition.

\begin{proposition}\label{proposition_gap_lower_bound_using_nominal_gap}
Given $\Sigma_{i}(\theta_{i}) \sim \mathbf{f}_{i}$ and the nominal models $\bar{\Sigma}_{i}$ for both systems $i \in \{1, 2\}$, consider a joint distribution $\pi_{\mathbf{f}} \in \Pi(\mathbf{f}_{1}, \mathbf{f}_{2})$. Then, $\forall q \geq 1$, we see that
\begin{align}
\label{eqn_proposition_gap_lower_bound_using_nominal_gap}
\mathrm{dist}_{\Sigma_{1}, \Sigma_{2}, \delta_{g}}^{q} 
&\geq 
\inf_{\pi_{\mathbf{f}} \in \Pi(\mathbf{f}_{1}, \mathbf{f}_{2})}
\left(
\mathrm{dist^{nom}_{\Sigma_{1}, \Sigma_{2}}} - 
\sum^{2}_{i = 1} \mathbb{E}_{\pi_{\mathbf{f}}}[\mathrm{Gap}_{i}(\theta_{i})]
\right)^{q}_{+}.
\end{align}
\end{proposition}
\begin{proof}
Fix any joint distribution $\pi_{\mathbf{f}} \in \Pi(\mathbf{f}_{1}, \mathbf{f}_{2})$. For any $\Sigma_{i}(\theta_{i}) \sim \mathbf{f}_{i}$ for $i \in \{1, 2\}$, we apply the triangle inequality to the gap metric to see that
\begin{align*}
\delta_{g}(\Sigma_{1}, \Sigma_{2}) 
\geq 
\left(
\mathrm{dist^{nom}_{\Sigma_{1}, \Sigma_{2}}} - 
\sum^{2}_{i = 1} \mathrm{Gap}_{i}(\theta_{i})
\right)_{+}.
\end{align*}
Then, given $q \geq 1$, we see that
\begin{align*}
\delta_{g}(\Sigma_{1}, \Sigma_{2})^{q} 
\geq 
\left(
\mathrm{dist^{nom}_{\Sigma_{1}, \Sigma_{2}}} - 
\sum^{2}_{i = 1} \mathrm{Gap}_{i}(\theta_{i})
\right)^{q}_{+}.
\end{align*}
Taking expectation on both sides with respect to the joint distribution $\pi_{\mathbf{f}}$, we see that
\begin{align*}
&\int_{\mathrm{supp}(\mathbf{f}_{1}) \times \mathrm{supp}(\mathbf{f}_{2})}  \delta_{g}(\Sigma_{1}, \Sigma_{2})^{q} \, \pi_{\mathbf{f}}(d \Sigma_{1}, d \Sigma_{2}) \\ 
&\geq
\mathbb{E}_{\pi_{\mathbf{f}}} 
\left[
\left(
\mathrm{dist^{nom}_{\Sigma_{1}, \Sigma_{2}}} - 
\sum^{2}_{i = 1} \mathrm{Gap}_{i}(\theta_{i})
\right)^{q}_{+}
\right] \\
&\geq
\left(
\mathrm{dist^{nom}_{\Sigma_{1}, \Sigma_{2}}} - 
\sum^{2}_{i = 1} 
\mathbb{E}_{\pi_{\mathbf{f}}} 
\left[
\mathrm{Gap}_{i}(\theta_{i})
\right]
\right)^{q}_{+},
\end{align*}
where we applied the Jensen's inequality (for the convex function $f(x) = (x)^{q}_{+}$ with $q \geq 1$, $\mathbb{E}[(x)^{q}_{+}] \geq (\mathbb{E}[x])^{q}_{+}$) to get the second inequality. Taking infimum over all joint distributions $\pi_{\mathbf{f}} \in \Pi(\mathbf{f}_{1}, \mathbf{f}_{2})$ on both sides yields the desired result \eqref{eqn_proposition_gap_lower_bound_using_nominal_gap} and the proof is complete. 
\end{proof}

\subsection{Numerical Demonstrations}
We considered the following state space models $(A_1, B_1, C_1, D)$ and $(A_2, B_2, C_2, D)$ as the nominal models for two different dynamical systems:
\begin{align} 
A_{1} 
&= 
\begin{bmatrix}
0 & 1 \\ -2 & -0.5   
\end{bmatrix},
B_{1} 
= 
\begin{bmatrix}
0 \\ 1    
\end{bmatrix},
C_{1}
=
\begin{bmatrix}
1 & 0    
\end{bmatrix},
D 
= 
0 \\
A_{2} 
&= 
\begin{bmatrix}
-3.2178 & 1.2354 \\ -1.7812 & -2.6507   
\end{bmatrix},
B_{2} 
= 
\begin{bmatrix}
0 \\ 1    
\end{bmatrix},
C_{2}
=
\begin{bmatrix}
1 & 0    
\end{bmatrix}. 
\end{align}
For generating the perturbed models, we choose $d = 4$, $\theta_{1} \sim \mathcal{N}(0.01, 0.01^2)$, and $\theta_{2} \sim \mathcal{N}(0.05, 0.05^2)$. Totally, $N = 50$ samples of perturbed models for each of the two dynamical systems were generated along the lines of \eqref{eqn_lemma_affine_parametric_plant_matrices}. The proposed distance measure $\mathrm{dist}_{\Sigma_{1}, \Sigma_{2}, \delta_{g}}$ between each of the models for both the systems was computed using \eqref{eqn_distance_between_distributions_system_models} using linear programming based approach with the transport cost being the gap metric which was computed using the \emph{gapmetric} command of Matlab. The upper bound and lower bounds for the proposed distance measure were computed using Proposition \ref{proposition_gap_diameter_upper_bound} and Proposition  \ref{proposition_gap_lower_bound_using_nominal_gap} respectively. The gap between the nominal models $\mathrm{dist^{nom}_{\Sigma_{1}, \Sigma_{2}}}$ was found to be $0.7731$. We estimated the proposed distance measure $\mathrm{dist}_{\Sigma_{1}, \Sigma_{2}, \delta_{g}} = 0.7765$, and its lower and upper bounds as $0.6561$ and $0.8252$ respectively.   

\subsection{Summary} 
The proposed time domain distance given by \eqref{eqn_distance_between_distributions_system_models} facilitates a simple linear programming based computation and also is devoid of additional supremum over frequency operation. We do not claim here that the upper and lower bounds given by Proposition \ref{proposition_gap_diameter_upper_bound} and Proposition  \ref{proposition_gap_lower_bound_using_nominal_gap} respectively are tight. In the next section, we will show that the frequency domain distance proposed in Section \ref{sec:frequency_domain} never exceeds the time domain distance proposed in this section. 

\section{Comparing Time Domain \& Frequency Domain Distance Measures} \label{sec_comparison}
It is well known from \cite{vinnicombe_tac_1993, vinnicombe2001uncertainty} that $\nu$-gap can never exceed the gap metric for linear systems. Along those lines, we will prove in this section that the proposed frequency domain distance measure in Section \ref{sec:frequency_domain} between two stochastic LTI dynamical system never exceeds its time domain distance measure counterpart proposed in Section \ref{sec:time_domain}. We formalise this observation using a comparison theorem followed by a simulation example based demonstration to corroborate our findings. \\

In the frequency domain setting, we know that the distribution $\mathbb{P}_{R_{i}(\omega)}$ governing the uncertainty of system $i$ with $i \in \{1, 2\}$ on the boundary of the Riemann sphere is related to the distribution $\mathbf{f}_{\theta_{i}}$ of the random parameter $\theta_{i}$ through \eqref{eqn_distribution_pushfwd_Riemannsphere}. Analogously, in the time domain setting, the distribution $\mathbf{f}_{i}$ of plant models of system $i \in \{1, 2\}$ and the corresponding distribution $\mathbf{f}_{\theta_{i}}$ of the random parameter $\theta_{i}$ satisfy \eqref{eqn_distribution_system_models} and \eqref{eqn_pushforward_measure} with $\Phi_{i}$ denoting the measurable map from the parameter space to the state space of system plants as described earlier in Section \ref{sec:time_domain}. We now define a transfer function mapping which when given a state space model, returns a real rational transfer function. That is, we define the transfer function mapping $\mathbf{TF}: \Phi_{i} \rightarrow \mathbf{R}\mathcal{L}_{\infty}$ such that $\mathbf{TF}(\Sigma_{i}(\theta_{i})) = P_{i}(\theta_{i}; s)$. We also need a mapping $\Psi^{\mathfrak{R}}_{\omega}: \Phi_{i} \rightarrow \partial \mathfrak{R}$ that takes the state space model and maps it to the Riemann sphere after realising a real rational transfer function and subsequently evaluating it at a particular frequency $\omega$ and applying the inverse stereographic projection operation. Such a mapping can be defined using composition as 
\begin{align}
\Psi^{\mathfrak{R}}_{\omega} 
=
\phi^{-1} \circ \mathcal{E}_{\omega} \circ \mathbf{TF}.     
\end{align}

\begin{assumption} \label{assume_mapping_continuity}
The mapping $\Psi^{\mathfrak{R}}_{\omega}: \Phi_{i} \rightarrow \partial \mathfrak{R}$ is continuous.
\end{assumption}
Before we proceed ahead with the comparison theorem, we will first prove a lemma describing how the joint distributions involved in optimal transport defined in the state space and in the Riemann sphere are related to each other and this will be useful in the proof of the comparison theorem to be presented later in this manuscript.
\begin{lemma}
\label{lemma_joint distribution_transport}
For every frequency $\omega \in \Omega$, the corresponding joint distribution $\pi_{\omega} \in \Pi(\mathbb{P}_{R_{1}(\omega)}, \mathbb{P}_{R_{2}(\omega)})$ on the Riemann sphere with marginals $\mathbb{P}_{R_{1}(\omega)}$ and $\mathbb{P}_{R_{2}(\omega)}$ is related to the joint distribution $\pi_{\mathbf{f}} \in \Pi(\mathbf{f}_{1}, \mathbf{f}_{2})$ defined in the space of systems with marginals $\mathbf{f}_{1}$ and $\mathbf{f}_{2}$ respectively as 
\begin{align}
\label{eqn_joint_distribution_relation} 
\pi_{\omega}
:=
\left(
\Psi^{\mathfrak{R}}_{\omega} \times \Psi^{\mathfrak{R}}_{\omega}
\right)_{\#} \pi_{\mathbf{f}}.
\end{align}
\end{lemma}
\begin{proof}
From the definition in \eqref{eqn_pushforward_definition}, we observe that for $\mathbf{A}_1 \times \mathbf{A}_{2} \subseteq \partial\mathfrak{R} \times \partial\mathfrak{R}$, 
\begin{align*}
\pi_{\omega}(\mathbf{A}_1 \times \mathbf{A}_{2})
&=
\pi_{\mathbf{f}}
\left(
\left(
\Psi^{\mathfrak{R}}_{\omega} \times \Psi^{\mathfrak{R}}_{\omega}
\right)^{-1}(\mathbf{A}_1 \times \mathbf{A}_{2})
\right) \\
&=
\pi_{\mathbf{f}}
\left(
\left(
\Psi^{\mathfrak{R}}_{\omega} 
\right)^{-1}
(\mathbf{A}_1)
\times 
\left(
\Psi^{\mathfrak{R}}_{\omega} 
\right)^{-1}
(\mathbf{A}_2)
\right)
\end{align*}
Now let $\mathbf{A}_1 = \mathbf{A} \subset \partial\mathfrak{R}$ and $\mathbf{A}_2 = \partial\mathfrak{R}$. Then,
\begin{align*}
\pi_{\omega}(\mathbf{A} \times \partial\mathfrak{R})
&=    
\pi_{\mathbf{f}}
\left(
\left(
\Psi^{\mathfrak{R}}_{\omega} 
\right)^{-1}
(\mathbf{A})
\times 
\left(
\Psi^{\mathfrak{R}}_{\omega} 
\right)^{-1}
(\partial\mathfrak{R})
\right) \\
&=    
\pi_{\mathbf{f}}
\left(
\left(
\Psi^{\mathfrak{R}}_{\omega} 
\right)^{-1}
(\mathbf{A})
\times 
\Phi_{2}
\right) \\
&=
\mathbf{f}_{1}
\left(
\left(
\Psi^{\mathfrak{R}}_{\omega} 
\right)^{-1}
(\mathbf{A})
\right) \\
&=
\left(
\Psi^{\mathfrak{R}}_{\omega} 
\right)_{\#}
\mathbf{f}_{1}
(\mathbf{A}) \\
&=
\mathbb{P}_{R_{1}(\omega)}(\mathbf{A}).
\end{align*}
By similar arguments, we will also get $\pi_{\omega}(\partial\mathfrak{R} \times \mathbf{A}) = \mathbb{P}_{R_{2}(\omega)}(\mathbf{A})$. Since we get the corresponding marginals, the result \eqref{eqn_joint_distribution_relation} follows immediately.
\end{proof}
In the following lemma, we will prove a similar result connecting the support sets of distributions living on Riemann sphere and the corresponding supports sets of distributions in the space of systems.
\begin{lemma}
\label{lemma_support-pushforward}
Let $\mathbf{f}_{i}$ be the distribution governing the system $i \in \{1, 2\}$ given by \eqref{eqn_perturbed_dynamics_system} in the state space. Since $\mathbb{P}_{R_{i}(\omega)} = \left(\Psi^{\mathfrak{R}}_{\omega} \right)_{\#} \mathbf{f}_{i}$ at frequency $\omega \geq 0$ by \eqref{eqn_joint_distribution_relation}, we get 
\begin{equation}\label{eq:supp-closure}
\mathrm{supp}\left(\mathbb{P}_{R_{i}(\omega)}\right) 
=
\mathcal{R}_{P_{1}(\omega)}
=
\overline{\,\Psi^{\mathfrak{R}}_{\omega}\bigl(\mathrm{supp}\left(\mathbf{f}_{i}\right)\bigr)\,}.
\end{equation}
Additionally, if $\mathrm{supp}\left(\mathbf{f}_{i}\right)$ is compact, then $\Psi^{\mathfrak{R}}_{\omega}(\mathrm{supp}\left(\mathbf{f}_{i}\right))$ is also compact and
so 
\begin{equation}
\label{eq:supp-equality-compact}
\mathrm{supp}\left(\mathbb{P}_{R_{i}(\omega)}\right) \ =\ \Psi^{\mathfrak{R}}_{\omega}\bigl(\mathrm{supp}\left(\mathbf{f}_{i}\right)\bigr).
\end{equation}
On the other hand, given distributions $\mathbb{P}_{R_{i}(\omega)}$ on $\partial\mathfrak{R}$ and $\mathbf{f}_{i}$ on the state space satisfying $(\Psi^{\mathfrak{R}}_{\omega})_{\#}\mathbf{f}_{i}=\mathbb{P}_{R_{i}(\omega)}$, the following inclusion for their support sets always holds:
\begin{equation}\label{eq:backward-inclusion}
\mathrm{supp}\left(\mathbf{f}_{i}\right) 
\subseteq \left(\Psi^{\mathfrak{R}}_{\omega}\right)^{-1}
\left(
\mathrm{supp}\left(\mathbb{P}_{R_{i}(\omega)}\right)
\right).
\end{equation}
\end{lemma}

\begin{proof}
To prove the forward inclusion of \eqref{eq:supp-closure}, let $\Sigma_{i} \in \mathrm{supp}\left(\mathbf{f}_{i}\right)$ and let $r_{i}(\omega) = \Psi^{\mathfrak{R}}_{\omega}(\Sigma_{i})$. Take an open neighbourhood $\mathcal{B}_{r_{i}(\omega)} \subset \partial\mathfrak{R}$ such that $r_{i}(\omega) \in \mathcal{B}_{r_{i}(\omega)}$. By continuity of $\Psi^{\mathfrak{R}}_{\omega}$ via Assumption \ref{assume_mapping_continuity}, there exists an open neighbourhood $\mathcal{B}_{\Sigma_{i}}$ of $\Sigma_{i}$ in the state space so that $\Psi^{\mathfrak{R}}_{\omega}(\mathcal{B}_{\Sigma_{i}}) \subseteq \mathcal{B}_{r_{i}(\omega)}$. Since $\Sigma_{i} \in \mathrm{supp}\left(\mathbf{f}_{i}\right)$, we have $\mathbf{f}_{i}(\mathcal{B}_{\Sigma_{i}})>0$, hence
\begin{align*}
\mathbb{P}_{R_{i}(\omega)}(\mathcal{B}_{r_{i}(\omega)}) 
= 
\mathbf{f}_{i} 
\left(
\left(
\Psi^{\mathfrak{R}}_{\omega}
\right)^{-1}
(\mathcal{B}_{r_{i}(\omega)})
\right)
\geq 
\mathbf{f}_{i}(\mathcal{B}_{\Sigma_{i}}) 
> 
0.
\end{align*}
Therefore $r_{i}(\omega) \in \mathrm{supp}\left(\mathbb{P}_{R_{i}(\omega)}\right)$, proving $\Psi^{\mathfrak{R}}_{\omega}(\mathrm{supp}\left(\mathbf{f}_{i}\right))\subseteq \mathrm{supp}\left(\mathbb{P}_{R_{i}(\omega)}\right)$. For the reverse inclusion, let $r_{i}(\omega) \in \partial\mathfrak{R} \setminus \overline{\Psi^{\mathfrak{R}}_{\omega}(\mathrm{supp}\left(\mathbf{f}_{i}\right))}$. Then, there is an open neighbourhood $\mathcal{B}_{r_{i}(\omega)}$ of $r_{i}(\omega)$ such that $\mathcal{B}_{r_{i}(\omega)} \cap \Psi^{\mathfrak{R}}_{\omega}(\mathrm{supp}\left(\mathbf{f}_{i}\right)) = \emptyset$. By continuity of $\Psi^{\mathfrak{R}}_{\omega}$ through Assumption \ref{assume_mapping_continuity}, we see that $\left(\Psi^{\mathfrak{R}}_{\omega}\right)^{-1}(\mathcal{B}_{r_{i}(\omega)})$ is open in the state space $\Phi_{i}$ and $\left(\Psi^{\mathfrak{R}}_{\omega}\right)^{-1}(\mathcal{B}_{r_{i}(\omega)}) \cap \mathrm{supp}\left(\mathbf{f}_{i}\right) = \emptyset$, hence $\mathbf{f}_{i} 
\left(\left(\Psi^{\mathfrak{R}}_{\omega}\right)^{-1}(\mathcal{B}_{r_{i}(\omega)}) \right) = 0$. Thus $\mathbb{P}_{R_{i}(\omega)}(\mathcal{B}_{r_{i}(\omega)}) = \mathbf{f}_{i}\left( \left(\Psi^{\mathfrak{R}}_{\omega}\right)^{-1}(\mathcal{B}_{r_{i}(\omega)})\right) = 0$, which shows $r_{i}(\omega) \notin \mathrm{supp}\left(\mathbb{P}_{R_{i}(\omega)}\right)$. So, $\mathrm{supp}\left(\mathbb{P}_{R_{i}(\omega)}\right) \subseteq \overline{\Psi^{\mathfrak{R}}_{\omega}(\mathrm{supp}\left(\mathbf{f}_{i}\right))}$. 
Combining with the forward inclusion, the result \eqref{eq:supp-closure} follows immediately. Additionally, if $\mathrm{supp}\left(\mathbf{f}_{i}\right)$ is compact, then $\Psi^{\mathfrak{R}}_{\omega}(\mathrm{supp}\left(\mathbf{f}_{i}\right))$ is compact in $\partial\mathfrak{R}$ (due to continuity in Assumption \ref{assume_mapping_continuity}), and hence it is closed too and thereby \eqref{eq:supp-closure} reduces to \eqref{eq:supp-equality-compact}. For proving \eqref{eq:backward-inclusion}, let $\mathbf{f}_{i}$ be a distribution on the state space with $(\Psi^{\mathfrak{R}}_{\omega})_{\#}\mathbf{f}_{i} = \mathbb{P}_{R_{i}(\omega)}$. If $\Sigma_{i} \in \mathrm{supp}\left(\mathbf{f}_{i}\right)$, then by the above forward inclusion arguments, we have $\Psi^{\mathfrak{R}}_{\omega}(\Sigma_{i})\in \mathrm{supp}\left(\mathbb{P}_{R_{i}(\omega)}\right)$ which is equivalent to $\Sigma_{i} \in \left(\Psi^{\mathfrak{R}}_{\omega}\right)^{-1}(\mathrm{supp}\left(\mathbb{P}_{R_{i}(\omega)}\right))$. Hence, the result \eqref{eq:backward-inclusion} follows.
\end{proof}

For deterministic SISO LTI dynamical system represented by their transfer functions $P_1(s), P_2(s) \in \mathbf{R}\mathcal{L}_{\infty}$, the $\nu$-gap between them (provided both systems satisfy the winding number constraint as described in \cite{vinnicombe_tac_1993}) is given by 
\begin{subequations}
\label{eqn_nuGap_definition}
\begin{align}
\delta_{\nu}(P_1,P_2) 
&:=
\sup_{\omega}
\kappa(P_1(j\omega),P_2(j\omega)), \quad \text{where}, \\
\kappa(P_1(j\omega),P_2(j\omega))
&:=
\mathrm{d_{chord}}(\phi^{-1}(P_1(j\omega)), \phi^{-1}(P_2(j\omega))) \\
&=
\frac{|P_1(j\omega)-P_2(j\omega)|}{\sqrt{(1+|P_1(j\omega)|^2)(1+|P_2(j\omega)|^2)}}
\end{align}
\end{subequations}
denotes the pointwise-in-frequency gap between $P_1$ and $P_2$ and is exactly equal to the chordal distance between the two frequency-response points after inverse stereographic projection to $\partial\mathcal{R}$. Using this observation, we will now proceed ahead with the comparison theorem to formally establish the fact the proposed frequency domain distance that is described in Section \ref{sec:frequency_domain} through the definition given \eqref{eqn_empirical_distance_p1_p2} never exceeds the proposed time domain distance counterpart that is described in Section  \ref{sec:time_domain} through the definition given by \eqref{eqn_distance_between_distributions_system_models}.

\begin{theorem}
\label{theorem_distances_comparison}
Consider two stochastic dynamical systems, whose state space models $\Sigma_{i}(\theta_{i})$ for system $i \in \{1, 2\}$ are given by \eqref{eqn_perturbed_dynamics_system}. Further,
let the corresponding random transfer functions obtained from the respective state space models $\Sigma_{i}(\theta_{i})$ denoted by $P_{i}(\theta_{i}; s)$ be given by \eqref{eqn_stochastic_transfer_function}. Additionally, let the random parameter $\theta_{i}$ affecting the system $i \in \{1, 2\}$ be governed by the respective distribution $\mathbf{f}_{\theta_{i}}$. Subsequently, let the distribution $\mathbb{P}_{R_{i}(\omega)}$ governing the uncertainty of system $i$ with $i \in \{1, 2\}$ on the boundary of the Riemann sphere be defined using \eqref{eqn_distribution_pushfwd_Riemannsphere} and the distribution $\mathbf{f}_{i}$ that governs the uncertainty of $\Sigma_{i}(\theta_{i})$ be given by \eqref{eqn_distribution_system_models}. Then, for $q \geq 1$, 
\begin{align}
\label{eqn_distance_comparison_result}
\mathrm{d}_{q}(P_1, P_2) 
\leq 
\mathrm{dist^{q}_{\Sigma_{1}, \Sigma_{2}, \delta_{g}}}.
\end{align}
\end{theorem}

\begin{proof}
Let $\pi_{\mathbf{f}} \in \Pi(\mathbf{f}_1,\mathbf{f}_2)$ be a joint distribution defined on the state-space $\Phi_1 \times \Phi_2$. Then, the corresponding joint distribution on the (transfer function) $\mathbf{R}\mathcal{L}_{\infty}$ space with marginal distributions $\mathbf{f}_{P_1}$, $\mathbf{f}_{P_2}$ can be obtained through the push-forward operation of $\pi_{\mathbf{f}}$ under the $\mathbf{TF}$ mapping using \eqref{eqn_joint_dist_state_space_tf_space_relation} as
\begin{align}
\label{eqn_thm_comparison_interim_step_0}
\Pi(\mathbf{f}_{P_1}, \mathbf{f}_{P_2})
\ni
\pi_{\mathbf{P}} 
:= 
(\mathbf{TF}\times \mathbf{TF})_{\#}\,\pi_{\mathbf{f}}.
\end{align}
Further, for each frequency $\omega \geq 0$, the corresponding joint distribution on the Riemann sphere can be obtained using Lemma \ref{lemma_joint distribution_transport} as
\begin{align*}
\Pi\!\big(\mathbb{P}_{R_1(\omega)},\mathbb{P}_{R_2(\omega)}\big)
\ni
\pi_{\omega} 
:= \big(\Phi^{\mathfrak{R}}_{\omega}\times \Phi^{\mathfrak{R}}_{\omega}\big)_{\#} \pi_{\mathbf{f}}.
\end{align*}
By definition of the frequency domain distance in \eqref{eqn_distance_btw_p1_p2} involving the Wasserstein distance defined over the infimum over joint distributions with the transport cost computed using the chordal distance, we see that 
\begin{align}
\label{eqn_thm_comparison_interim_1}
&W_q^{q}
\big(
\mathbb{P}_{R_1(\omega)},\mathbb{P}_{R_2(\omega)}
\big) \nonumber \\
&\leq
\int_{\mathcal{R}_{P_{1}(\omega)} \times \mathcal{R}_{P_{2}(\omega)}} \mathrm{d_{chord}}(r_1(\omega), r_2(\omega))^{q} \, \pi_{\omega}(d r_1, d r_2) \nonumber \\
&= 
\int_{\mathrm{supp}(\mathbf{f}_{1}) \times \mathrm{supp}(\mathbf{f}_{2})} \mathrm{d_{chord}}
\big(
\Phi^{\mathfrak{R}}_{\omega}(\Sigma_1),\Phi^{\mathfrak{R}}_{\omega}(\Sigma_2)
\big)^{q} \, \pi_{\mathbf{f}}(d \Sigma_1, d \Sigma_2).
\end{align}
From \cite{vinnicombe2001uncertainty}, we know that for every pair of transfer functions $(P_1,P_2)$ in the $\mathbf{R}\mathcal{L}_{\infty}$ space, the following inequality holds due to \eqref{eqn_nuGap_definition}:
\begin{equation}
\label{eqn_thm_comparison_interim_2}
\footnotesize
\mathrm{d_{chord}} \big(\phi^{-1}(\mathcal{E}_\omega(P_1)),\phi^{-1}(\mathcal{E}_\omega(P_2))\big)
\leq 
\delta_\nu(P_1,P_2) 
\leq \delta_{g}(P_1,P_2).
\end{equation}
Applying $P_i = \mathbf{TF}(\Sigma_i)$ \& the inequality \eqref{eqn_thm_comparison_interim_2} in \eqref{eqn_thm_comparison_interim_1}, we get
\begin{align*}
&W_q^{q}
\big(
\mathbb{P}_{R_1(\omega)},\mathbb{P}_{R_2(\omega)}
\big) \\
&\leq 
\int_{\mathrm{supp}(\mathbf{f}_{1}) \times \mathrm{supp}(\mathbf{f}_{2})} \delta_{g}\!\big(\mathbf{TF}(\Sigma_1),\mathbf{TF}(\Sigma_2)\big)^{q} \pi_{\mathbf{f}}(d \Sigma_1, d \Sigma_2) \\
&= \int_{\mathrm{supp}\left(\mathbf{f}_{P_{1}}\right) \times \mathrm{supp}\left(\mathbf{f}_{P_{2}}\right)} \delta_{g}(P_1,P_2)^{q}\, \pi_{\mathbf{P}}(d P_1, d P_2),
\end{align*}
where we applied \eqref{eqn_thm_comparison_interim_step_0} to get the last equality. Now taking supremum over $\omega\in\mathbb{R}$ on the both sides, we get
\begin{align*}
&\underbrace{\sup_{\omega \geq 0} W_q^{q}\big(\mathbb{P}_{R_1(\omega)},\mathbb{P}_{R_2(\omega)}}_{= \mathrm{d}_q(P_1,P_2)}
\big) \\
&\leq
\sup_{\omega \geq 0}
\int_{\mathrm{supp}\left(\mathbf{f}_{P_{1}}\right) \times \mathrm{supp}\left(\mathbf{f}_{P_{2}}\right)} \delta_{g}(P_1,P_2)^{q}\, \pi_{\mathbf{P}}(d P_1,d P_2) \\
&=
\int_{\mathrm{supp}\left(\mathbf{f}_{P_{1}}\right) \times \mathrm{supp}\left(\mathbf{f}_{P_{2}}\right)} \delta_{g}(P_1,P_2)^{q}\, \pi_{\mathbf{P}}(d P_1,d P_2),
\end{align*}
Taking infimum over all joint distributions $\pi_{\mathbf{P}}\in\Pi(\mathbf{f}_{P_{1}},\mathbf{f}_{P_{2}})$ on both sides, we get
\begin{align*}
&\underbrace{\inf_{\pi_{\mathbf{P}}\in\Pi(\mathbf{f}_{P_{1}},\mathbf{f}_{P_{2}})} \mathrm{d}_q(P_1,P_2) }_{= \mathrm{d}_q(P_1,P_2)}    \\
&\leq 
\inf_{\pi_{\mathbf{P}}\in\Pi(\mathbf{f}_{P_{1}},\mathbf{f}_{P_{2}})}   
\int_{\mathrm{supp}\left(\mathbf{f}_{P_{1}}\right) \times \mathrm{supp}\left(\mathbf{f}_{P_{2}}\right)} \!\!\delta_{g}(P_1,P_2)^{q}\, \pi_{\mathbf{P}}(d P_1,d P_2) \\
&=
\mathrm{dist}^{q}_{\Sigma_1,\Sigma_2,\delta_{g}}.
\end{align*}
This completes the proof.
\end{proof}

\subsection*{Numerical Demonstration}
Given two stochastic LTI systems, to demonstrate that the frequency domain distance never exceeds the time domain distance, we consider two different second order LTI systems which vary due to their correspondingly random damping and resonant frequency values. Specifically, consider two second order systems whose nominal damping and resonant frequency values are given by $\bar{\zeta}_{1} = 0.35$, $\bar{\zeta}_{2} = 0.55$, $\bar{\omega}_{n_{1}} = 1.8$, $\bar{\omega}_{n_{2}} = 1.2$ respectively. Then, the nominal transfer function models of both the systems are given by 
\begin{align}
\label{eqn_nominal_models_comparison_simulation}
\bar{P}_1(s) 
=     
\frac{1}{s^{2} + 1.26s + 3.24},
\quad 
\bar{P}_2(s) 
=     
\frac{1}{s^{2} + 1.32s + 1.44}.
\end{align}
To compute the distances, $N = 100$ samples of perturbed plant models for both the systems were formed by perturbing along the lines of \eqref{eqn_lemma_affine_parametric_plant_matrices}, the corresponding nominal models of both the systems given by \eqref{eqn_nominal_models_comparison_simulation}. The random parameters $\theta_{i}$ that were used to generate the perturbed models of the system $i \in \{1, 2\}$ are given by $\theta_{i} \sim \mathcal{N}(\mu_{\theta_{i}}, \Sigma_{\theta_{i}})$, where
\begin{align*}
\mu_{\theta_{1}} 
&=
\begin{bmatrix}
0.10 \\ -0.05 \\ 0.02  
\end{bmatrix},  
\Sigma_{\theta_{1}}
=
\begin{bmatrix}
0.15^{2} & 0 & 0 \\
0 & 0.20^{2} & 0 \\
0 & 0 & 0.10^{2}
\end{bmatrix}, \\
\mu_{\theta_{2}} 
&=
\begin{bmatrix}
-0.08 \\ 0.06 \\ -0.01  
\end{bmatrix},
\Sigma_{\theta_{2}}
=
\begin{bmatrix}
0.12^{2} & 0 & 0 \\
0 & 0.18^{2} & 0 \\
0 & 0 & 0.08^{2}
\end{bmatrix}.
\end{align*}
To compute the frequency domain distance, a frequency grid in the log space between $[10^{-2}, 10^{2}]$ rad/sec was discretized into $M = 100$ points. Precisely speaking, $\Omega_{M} = \mathrm{logspace}(10^{-2}, 10^{2}, 100)$. The quantities of interests namely the frequency domain distance $\mathrm{d}_{q}(P_1, P_2)$ and the time domain distance $\mathrm{dist}_{\Sigma_{1}, \Sigma_{2}, \delta_{g}}$ were computed using \eqref{eqn_distance_btw_p1_p2} and \eqref{eqn_distance_between_distributions_system_models} respectively. In both the distance computations, the corresponding type-$1$ Wasserstein distance computation was carried out using the linear programming technique. As a result of the computation, we obtained the frequency domain distance $\mathrm{d}_{q}(P_1, P_2) = 0.3795$, the time domain distance $\mathrm{dist}_{\Sigma_{1}, \Sigma_{2}, \delta_{g}} = 0.3812$ and the gap metric between the nominal models given by \eqref{eqn_nominal_models_comparison_simulation} was found out to be $0.3822$. Clearly as expected, we obtained $\mathrm{d}_{q}(P_1, P_2) \leq \mathrm{dist}_{\Sigma_{1}, \Sigma_{2}, \delta_{g}}$ and thereby agreeing to the claims of Theorem \ref{theorem_distances_comparison}. 
\section{Conclusions}
\label{sec:conclusions}
A new distance metric between two SISO stochastic LTI dynamical systems was presented both in the frequency domain and in the time domain. In the frequency domain, the proposed distance corresponds to the worst-case-in-frequency chordal distance metric induced distance between distributions characterising the uncertainties of systems in the Riemann sphere. Analogously, the proposed distance in the time domain corresponds to the gap metric induced type-q Wasserstein distance between the push-forward measures under both systems' corresponding measurable maps from the parameter space to their respective space of system plants. For both the frequency domain and the time domain settings, upper bounds and lower bounds for the proposed distances were given. It was also shown that for stochastic LTI systems, the proposed frequency domain distance measure never exceeds the proposed time domain distance measure counterpart.\\

There are several promising future  directions to expand this research and some of them are listed below:
\begin{itemize}
    \item Seek to extend the study from SISO systems to MIMO systems and further to nonlinear systems
    \item A preliminary step towards the above extension would be to investigate the probabilistic robustness for linear time varying (LTV) systems by adopting the ideas of \cite{cantoni2013SIAM, khong2014gaptimevarying} and adding probabilistic rigour on top of it and extending it to distance between stochastic LTV systems as done in this manuscript. 
    \item Another interesting direction of research will be to investigate the probabilistic robust stability of controllers. That is, given a stabilising controller for one stochastic system, we should investigate the probability of that controller stabilising another stochastic system in the vicinity of the first stochastic system where the vicinity is measured using the proposed distance measure.
    \item It would also be interesting to study probabilistic guarantees on the performance variations for the same controller trying to control two different stochastic systems. 
\end{itemize}



\bibliographystyle{IEEEtran}
\bibliography{bibliography}

\begin{IEEEbiography}[{\includegraphics[width=1in,height=1.25in,clip,keepaspectratio]{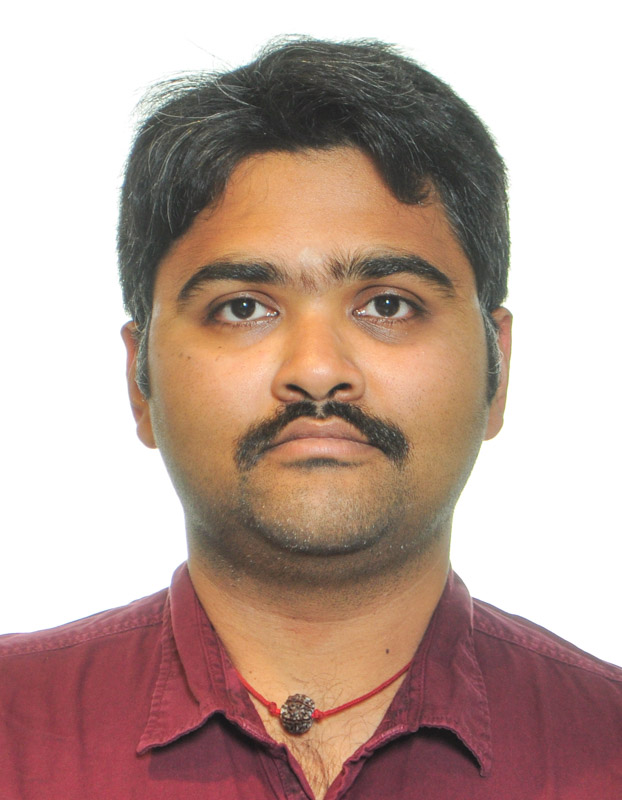}}]{Venkatraman Renganathan} (Member, IEEE) received his Bachelors degree in electrical and electronics engineering from the Government College of Technology, Anna University, Coimbatore, India in 2011. Further, he obtained his Masters degree in electrical engineering with focus on control systems from Arizona State University, USA, in 2016 and the Ph.D. degree in mechanical engineering with emphasis on dynamics and control from The University of Texas at Dallas, USA in 2021. He was a postdoctoral fellow at the department of automatic control at Lund University in Sweden from August 2021 till August 2024. Currently, he is a lecturer at the faculty of engineering \& applied sciences in Cranfield University, United Kingdom. His research interests include probabilistic robust control, distributed adaptive control, and risk bounded motion planning.
\end{IEEEbiography}

\begin{IEEEbiography}
{Sei Zhen Khong} (Senior Member, IEEE) received the Bachelor of Electrical Engineering degree (with first class honours) and the Ph.D. degree in Electrical Engineering from the University of Melbourne, Australia, in 2008 and 2012, respectively. He held research positions at the Department of Electrical and Electronic Engineering, University of Melbourne, Australia, Department of Automatic Control, Lund University, Sweden, Institute for Mathematics and its Applications, University of Minnesota Twin Cities, USA, and Department of Electrical and Electronic Engineering, University of Hong Kong, China. He is now with the Department of Electrical Engineering, National Sun Yat-sen University, Taiwan. His research interests include network control, systems theory, and optimization with applications to systems biology, power systems, and electromechanical systems.
\end{IEEEbiography}

\end{document}